\documentclass[12pt,a4paper]{article}

\usepackage[utf8]{inputenc}
\usepackage[T1]{fontenc}
\usepackage{amsmath}
\usepackage{amsfonts}
\usepackage{amssymb}
\usepackage{makeidx}
\usepackage{graphicx}
\usepackage{csquotes}
\usepackage{hyperref}
\usepackage[english]{babel}
\usepackage{amsthm}
\usepackage{bbold}
\usepackage{cases}
\usepackage{enumitem}
\usepackage{fullpage} % Agrandit les dimensions du texte (hauteur, largeur,
\usepackage{graphicx}
\usepackage{empheq}

\usepackage[title]{appendix}

\newtheorem{thm}{Theorem}[section]    % numérotés par section
\newtheorem{prop}[thm]{Proposition}    % Les propositions ont le même compteur
\newtheorem{defi}{Definition}[section]                             % que les théorèmes
\newtheorem{cor}[thm]{Corollary}    % Les propositions ont le même compteur
\newtheorem{lemme}[thm]{Lemma}   

\newtheorem{NB}{\textbf{\underline{Remark}}}
 
\newcommand\hh{\mathbb{H}}
\newcommand\bbb{\mathbb{B}}
\newcommand{\pr}{\mathbb{P}}
\newcommand{\esp}{\mathbb{E}}

\newcommand{\eps}{\epsilon}
\newcommand{\ga}{\gamma}
\newcommand{\A}{\mathcal{A}}

\newcommand{\ad}{\mbox{ad}}
\newcommand{\ph}{\varphi}
\DeclareMathOperator{\Hess}{Hess}
\newcommand{\sgn}{\textrm{sign}}
\DeclareMathOperator{\Span}{Span}

\title{Couplings of Brownian motions on {$SU(2)$} and {$SL(2,\mathbb{R})$}}
\usepackage{authblk}
\date{}                     %% if you don't need date to appear
\author[1]{Magalie Bénéfice}
\affil[1]{{Univ. Bordeaux, CNRS, Bordeaux INP, IMB, UMR 5251},{Talence},
	{F-33400}, 
	{France}}
\begin{document}
	
	\maketitle

\begin{abstract}
% The Lie groups $SU(2)$ and $SL(2,\mathbb{R})$ are some space models in subelliptic geometry. 
% Similar to the Heisenberg group, the subelliptic Brownian motion on $SU(2)$ (resp. $SL(2,\mathbb{R})$) may be seen as a Brownian motion on the sphere (resp. the hyperbolic plane) together with its swept area. The main aim of this article is to propose an explicit construction of a co-adapted successful coupling on $SU(2)$. The strategy is an alternation between reflection coupling and synchronous (with noise) coupling on the sphere. Using the above geometrical interpretation, we show how to evaluate the subRiemannian distance between the two induced subelliptic Brownian motions. With similar constructions, we also describe and study the subRiemannian distance between general co-adapted couplings in both $SU(2)$ and $SL(2,\mathbb{R})$.\\

The Lie groups $SU(2)$ and $SL(2,\mathbb{R})$ can be viewed as model spaces in subRiemannian geometry. 
Coupling two subelliptic Brownian motions on $SU(2)$ (resp. $SL(2,\mathbb{R})$) consists in coupling two Brownian motions on the sphere (resp. the hyperbolic plane) and simultaneously their swept areas. Using this approach we propose an explicit construction of a co-adapted successful coupling on $SU(2)$. The strategy is to alternate between reflection and synchronous (with noise) coupling on the sphere. We also describe some more general constructions of co-adapted couplings on $SU(2)$ and also on $SL(2,\mathbb{R})$.
\end{abstract}


%%\item Research highlight 1
%%\item Research highlight 2


\section{Introduction}\label{sec: Introduction}
% The aim of this article is to construct and study co-adapted couplings of Brownian motion on the subRiemannian manifolds $SU(2)$ and $SL(2,\mathbb{R})$. In particular we look for co-adapted successful couplings on $SU(2)$. 
Let us first recall some of the most important concepts. Here, as we consider Brownian motions on a subRiemannian manifold $G$, a coupling of Brownian motion starting from $g$ and $g'\in G$ will be defined as a probability measure on $\mathcal{C}([0,T[,G\times G)$ (with $T\in[0,+\infty]$) whose first marginal is the law of a Brownian motion $(\bbb_t)_{t}$ starting from $g$ and whose second marginal is the law of a Brownian motion $(\bbb_t')_{t}$ starting from $g'$. In other words, we consider the joint law of $(\bbb_t,\bbb_t')_{t}$.  Here we will focus on co-adapted couplings, i.e., couplings such that the future of the processes only depends on their common past. Note that this notion generalises the notion of Markovian coupling for which the process $(\bbb_t,\bbb_t')_{t}$ is Markovian. We also interest ourselves in finding co-adapted successful couplings, that is, co-adapted couplings for which the processes meet at an almost surely finite time. In the case of $SU(2)$, we prove the existence of such a coupling. This successful coupling is the first one constructed on $SU(2)$ and is explicitly described.

In order to do that, we develop some new geometrical interpretations of the cylindrical coordinates on $SU(2)$ and $SL(2,\mathbb{R})$. Then the main strategy is to couple the driving noises, which are the projections of $\bbb_t$ on some Riemannian manifolds of constant curvature. Note that, after the completion of this work, we have continued the study of successful couplings on these two subRiemannian manifolds by using non co-adapted couplings. This work can be found in~\cite{Nonco-adaptedSU(2)}. More recently, Luo and Neel (\cite{luo2024nonmarkovian}) also looked at non co-adapted couplings using a different strategy to construct efficient couplings (see the definition in Subsection \ref{subsec: Motivation}). Although the present work, which uses co-adapted couplings, does not aim to obtain such efficient couplings, it contributes to the general knowledge of hypoelliptic diffusions. We hope that this will enable future applications which can not be achieved by using efficient couplings. Note also that we announced some of the above results without proofs in the conference paper~\cite{GSI}.\\
 In the next subsection (\ref{subsec: Motivation}), we give an overview of the state of the art and the motivation for this work, in a second subsection (\ref{subsec: Main results}) we summarise our main results and in the last subsection (\ref{subsec: Organisation}) we describe the overall structure of the article.

\subsection{Motivation}\label{subsec: Motivation}
The notion of coupling has been quite developed these last decades, first for Markov chains, then for Markov processes (see~\cite{lindvall2002lectures} for a general introduction). 
The coupling of Brownian motions has been studied on $\mathbb{R}^n$~\cite{kendall-coupling-gnl} and also on Riemannian manifolds (see~\cite{pascu2018couplings}).
%In the two previous citations, the authors focus in particular on co-adapted couplings, which are couplings where Brownian motions are adapted to the same filtration and which are often quite easy to construct.
Aside from providing a better understanding of the geometry of the state space, couplings are essential for many analysis results involving the harmonic functions and the heat semi-group like Harnack, Poincaré, Sobolev or Wasserstein inequalities (see~\cite{kuwada,WangBInequalities,CranstonRefl} for some examples). In the previous citations, the studied diffusions are elliptic. The case of hypoelliptic diffusions is a current field of research: as an example, the kinetic Langevin diffusion is studied in~\cite{Eberle}, Kolmogorov type diffusions are dealt with in~\cite{CranstonKolmogorov,Kolmogorov}, Brownian motions on the Heisenberg group can be found in~\cite{CranstonKolmogorov,Kolmogorov,kendall-coupling-gnl,kendall2009brownian,kendall2007coupling,banerjee2017coupling,bonnefont2018couplings,CoAdaptSuccessHeisenberg2}. 
% In particular, the coupling methods are used to obtain analysis inequalities in~\cite{Kolmogorov,banerjee2017coupling,bonnefont2018couplings}, offering alternatives to more analytic methods (~see  
In the above examples, the diffusions are of the form $(X_t,z_t)$ with $X_t$ an elliptic diffusion that we call here "driving noise" and with $z_t=f((X_s)_{s\leq t})$ for $f$ a functional (for other examples of diffusion processes of this form, see~\cite{baudoin2023stochastic}). Then, the difficulty is that we can only couple the driving noises $(X_t,X_t')_t$. 

We interest ourselves in Brownian motions on subRiemannian manifolds. In our examples, the Brownian motion is the diffusion process whose infinitesimal generator is the canonical subelliptic subLaplacian operator. As, in general, the natural subRiemannian distance lacks of smoothness, Itô's formula can not be directly applied to compare the processes. When the subRiemannian manifold can be obtained from a Sasakian contact manifold, one idea could be to deal with some smooth Riemannian metrics converging to the subRiemannian metric as in~\cite{baudoin2022variations}. The usual strategy~\cite{kendall2007coupling,banerjee2017coupling,kendall-coupling-gnl,CoAdaptSuccessHeisenberg2} in the Heisenberg group is different. In this case, $X_t$ is a two dimensional Brownian motion and $z_t$ is the Levy swept area which is, up to a constant the signed area swept by $(X_s)_{s\leq t}$ and by the geodesics joining $X_0$ (resp. $X_t$) to the origin. Moreover the subRiemannian distance $d_{cc}$ satisfies:
\begin{equation}\label{equivHeisenberg}
    c_1\left(R_t+\sqrt{|A_t|}\right)\leq d_{cc}\left((X_t,z_t),(X_t',z_t')\right)\leq c_2\left( R_t+\sqrt{|A_t|}\right)
\end{equation}
with $R_t=||X_t-X_t'||_2$ ($||\cdot||_2$ being the Euclidean norm on $\mathbb{R}^2$) and $A_t$ a signed swept area between the two driving noises. Here, $c_1$ and $c_2$ are some positive constants only depending on the structure of the Heisenberg group. Thanks to this comparison, it is possible to evaluate the distance between the two processes.\\
Our goal is to use a similar strategy to deal with couplings in some other SubRiemannian manifolds. We look at the two model spaces $SU(2)$ and $SL(2,\mathbb{R})$ that could play, together with the Heisenberg group, the role of "constant curvature subelliptic manifolds" in dimension $3$ (note that the difficulty to obtain a consistent definition of the curvature for subelliptic manifold is a real problem that we will not develop here, see~\cite{BakryLiYau,AgravechCurvature,FalbelCurvature} for some notions of curvature on these model spaces). Baudoin and Bonnefont proved in~\cite{baudoin2009subelliptic,bonnefont2012subelliptic,bonnefont-these} that the Brownian motion on $SU(2)$ (resp. $SL(2,\mathbb{R})$) can again be written as $(X_t,z_t)$ with $X_t$ a Brownian motion on the sphere $S^2$ (resp. on the hyperbolic plane $\mathbf{H}^2$) induced by the Hopf fibration and $z_t$ a signed swept area. (Note that a different decomposition has been obtained in~\cite{SL(2)Malliavin,AlbeverioSL(2)} for $SL(2,\mathbb{R})$ but for the elliptic Brownian motion.)  As the structure is similar to the one in the Heisenberg group, it seems reasonable to try to generalise some of the existing coupling methods, in particular to construct successful couplings for $SU(2)$. 

The successful couplings are the couplings for which the first meeting time (or "coupling time") $\tau:=\inf\{t>0|\bbb_t=\bbb_t'\}$ of the Brownian motions is almost surely finite. Even in Riemannian manifolds, such couplings do not always exist. For example there is no successful coupling in the hyperbolic plane (this will be briefly explained with reference in subsection \ref{subsec: Main results}). %(see~\cite{wang2002liouville} for some criterion of existence of successful couplings in Riemannian manifold)
Historically, a first motivation in constructing successful couplings is to obtain estimates on the total variation distance between the laws of the Brownian motions. Indeed for every coupling of Brownian motions $(\bbb_s,\bbb_s')_s$ and every $t>0$, we have: %\begin{equation}\label{Aldous}
$\pr(\tau>t)\geq d_{TV}(\mathcal{L}(\bbb_t),\mathcal{L}(\bbb'_t))$
%\end{equation}this quantity can be upper bounded by the rate of convergence $\pr(\tau>t)$ of the coupling 
(see the Coupling inequality in~\cite{asmussen2003applied}, Chapter VII).
Note that the couplings (not necessarily successful) for which this inequality becomes an equality are called maximal couplings. If it has been proved that such couplings always exist in the case of cadlag processes on Polish spaces~\cite{MaximalCouplingSverchkov}, they can be very difficult to study (explicit construction, simulation, estimation of a coupling rate) as one will often need some knowledge of the future of one of the process. For a Riemannian manifold having a kind of "reflection structure" just like the plane or the sphere, this can be done by using the reflection coupling (see~\cite{ReflKuwada,HsuSturmMaxEuc}). In this specific case the coupling is co-adapted and even Markovian. More generally, the existence of Markovian maximal couplings of regular elliptic diffusions has been studied in~\cite{BanerjeeMaxElliptic}. This existence depends on the same rigidity properties of the Riemannian manifold as well as on strong conditions on the drift part of the diffusion processes. In fact, Markovian maximal couplings are rare in Riemannian manifolds.
%, in the sense that the processes are adapted to the same filtration.
%
\\
On our model spaces, constructing a successful coupling means that we need to "couple" at an almost surely finite time the driving noises $X_t$ and $X_t'$ together with the swept areas $z_t$ and $z_t'$. In~\cite{kendall2007coupling}, Kendall constructed a successful co-adapted coupling on the Heisenberg group. In~\cite{banerjee2017coupling}, Banerjee, Gordina and Mariano proposed another coupling, not co-adapted but still successful. In fact, this coupling is efficient, i.e., the coupling rate $\pr(\tau>t)$ has the same order than the total variation distance $d_{TV}\left(\mathcal{L}(\bbb_t),\mathcal{L}(\bbb_t')\right)$ for $t$ large. In particular, if the driving noises start from the same point, the coupling rate of the non co-adapted coupling is better ($\pr(\tau>t)\lesssim \frac{1}{t}$) than for any co-adapted one ($\pr(\tau>t)\gtrsim \frac{1}{\sqrt{t}}$). Moreover, in~\cite{banerjee2017coupling}, the non co-adapted coupling leads to gradient estimates for harmonic functions just like the Cheng-Yau inequality. Note that, using a different approach, this last inequality has been proven in~\cite{GradientSansCouplageGordina} for an enlarged class of subRiemannian manifolds.\\
In the present article we deal with the generalisation of Kendall's co-adapted coupling.
%We propose the first construction of successful couplings of Brownian motion in the subRiemannian manifold $SU(2)$. 
Note that other co-adapted successful couplings on the Heisenberg group can be found in~\cite{kendall-coupling-gnl,CoAdaptSuccessHeisenberg2}. We did not work on their generalisations for the moment. 

As mentioned before, during the process of submission of this work, successful and efficient non co-adapted couplings on $SU(2)$ (and in a weaker sense on $SL(2,\mathbb{R})$) has been studied in~\cite{Nonco-adaptedSU(2)}, using a generalisation of~\cite{banerjee2017coupling}, and in~\cite{luo2024nonmarkovian}, with a new strategy of coupling. In the two cases the obtained coupling rate is exponentially decreasing.

% Another motivation is the study of harmonic functions. As an example, classical Liouville theorem says that every harmonic bounded function on $\mathbb{R}^2$ is constant. Wang proved in~\cite{wang2002liouville}, the following result:
%  \begin{thm}\label{Liouville}
%  Let $M$ be a Riemannian manifold with Ricci curvature bounded below. There exists a successful coupling  if and only if all harmonic bounded function on $M$ are constant.
%  \end{thm} 

\subsection{Main results}\label{subsec: Main results}
%In this subsection, we present our main results. 
%We first express more formally the geometrical interpretation of the Brownian motions on the two model spaces $SU(2)$ and $SL(2,\mathbb{R})$.
%\\
In the same way that we have a natural submersion from the Heisenberg group $\hh$ to the plane, in the case of $SU(2)$, one can define a submersion $\Pi_1$ induced by the Hopf fibration from $SU(2)$ to the sphere $S^2$. Similarly, one can define a submersion $\Pi_2$ from $SL(2,\mathbb{R})$ to the hyperbolic plane $\mathbf{H}^2$. We are going to consider some good coordinates called the cylindrical coordinates on our subRiemannian manifolds. 

In the first following result we use these cylindrical coordinates to compare (with respect to the subRiemannian structure) two elements on $SU(2)$ (resp. $SL(2,\mathbb{R})$) by comparing their projections on $S^2$ (resp. $\mathbf{H}^2$). 
\begin{prop}\label{prop: distance ponctuelle}
	Let $g=(\ph,\theta,z)$ and $g'=(\ph',\theta',z')$ be two elements of $SU(2)$ (resp. $SL(2,\mathbb{R})$) written in cylindrical coordinates. Using the submersions defined above, we denote $x:=\Pi_1(g)$ and $x':=\Pi_1(g')$ (resp. $x:=\Pi_2(g)$ and $x':=\Pi_2(g')$). The cylindrical coordinates of $g^{-1}\cdot g'$ are given by $(\rho,\Theta,\zeta)$ with for $\rho$ and $\zeta$:
	\begin{itemize}
		\item $\rho$ equals to the usual Riemannian distance on $S^2$ (resp. on $\mathbf{H}^2$) between $x$ and $x'$.
		\item $\zeta\in]-2\pi,2\pi]$ and $\zeta\equiv z'-z+\sgn(\theta-\theta')\mathcal{A}_{x',x,N_0}\mod(4\pi)$
		%$\zeta\equiv z'-z+\sgn(\theta-\theta')\mathcal{A}_{x,x',N_0}\mod(4\pi)$
		with $\mathcal{A}_{a,b,c}$ the area of the spherical (resp. hyperbolic) triangle of vertices $a, b$ and $c$ and $N_0$ a pole induced by the submersion $\Pi_1$ (resp. $\Pi_2$). Note that $\sgn(\theta-\theta')\mathcal{A}_{x',x,N_0}$ is in fact the signed area of the oriented triangle of vertices $x'$, $x$ and $N_0$.
	\end{itemize}
	In particular, we directly obtain the estimate:
	 \begin{equation}\label{eq: equivPonctuelle}
		c_1( \rho^2+|\zeta|)\leq d_{cc}^2(g,g')\leq c_2(\rho^2+|\zeta|)
	\end{equation} 
	with $c_1$ and $c_2$ two positive constants only depending on the subRiemannian structure of $SU(2)$ (resp. $SL(2,\mathbb{R})$).
\end{prop}

Let us now consider a Brownian motion $\bbb_t$ on $SU(2)$ (resp. $SL(2,\mathbb{R})$) and $(\ph_t,\theta_t,z_t)$ its cylindrical coordinates. Baudoin and Bonnefont proved in~\cite{baudoin2009subelliptic,bonnefont2012subelliptic,bonnefont-these} that:
\begin{itemize}
    \item $\left(X_t:=\Pi_1(\bbb_t)\right)_t$ (resp. $\Pi_2(\bbb_t)$) is a Brownian motion on $S^2$ (resp. on $\mathbf{H}^2$).
    \item $z_t-z_0$ is the signed swept area (modulo $4\pi$) of $X_t$ with respect to the fixed pole $N_0$.
\end{itemize}
In particular, these two processes entirely describe $\bbb_t$ and $\bbb_t$ is characterized by $(X_s)_{s\leq t}$. Note that this result is also written in Proposition \ref{distanceB} and explained in subsection \ref{subsec: Interpretation}.

From Proposition \ref{prop: distance ponctuelle}, we obtain:
\begin{cor}\label{distanceB2}
	Let us take $\bbb_t$ and $\bbb_t'$ two Brownian motions on $SU(2)$ (resp. $SL(2,\mathbb{R})$) starting from $g$ and $g'$ respectively, with $(\ph_t,\theta_t,z_t)$ and $(\ph'_t,\theta'_t,z'_t)$ their cylindrical coordinates.
	We denote $X_t:=\Pi_1(\bbb_t)$ (resp. $\Pi_2(\bbb_t)$) and $Y_t:=\Pi_1(\bbb_t')$ (resp. $\Pi_2(\bbb_t')$). The cylindrical coordinates of ${\left(\bbb_t\right)}^{-1}\cdot\bbb'_t$ are given by $(R_t,\Theta_t,\zeta_t)$ with:
	\begin{itemize}
		\item $R_t$ the Riemannian distance between $X_t$ and $Y_t$;
		\item $\zeta_t\in ]-2\pi,2\pi]$ and $\zeta_t\equiv A_t+z_0'-z_0+\sgn(\theta_0-\theta_0')\mathcal{A}_{Y_0,X_0,N_0}\mod(4\pi)$
		with $A_t$ the signed swept area between $(X_s)_{s\leq t}$ and $(Y_s)_{s\leq t}$ as defined in Definition \ref{def: aireBalayee1}.
	\end{itemize}
		As previously, we directly obtain the estimate:
	 \begin{equation}\label{eq: equivBrownien}
		c_1(R_t^2+|\zeta_t|)\leq d_{cc}^2(\bbb_t,\bbb_t')\leq c_2 (R_t^2+|\zeta_t|)
	\end{equation} 
	with $c_1$ and $c_2$ as in Proposition \ref{prop: distance ponctuelle}.
\end{cor}
With this Corollary we are able to compare two Brownian motions in the subRiemannian metric. Note that the relation (\ref{eq: equivBrownien}) is similar to the relation (\ref{equivHeisenberg}) obtained in the case of the Heisenberg group. 
% 	\begin{prop}
% 	Let us take $\bbb_t^x$ and $\bbb_t^{y}$ two Brownian motions on $SU(2)$ (resp. $SL(2,\mathbb{R})$) starting from $x$ and $y$ respectively.
% 	We denote $X_t=\pi_1(\bbb_t^x)$ (resp. $\pi_2(\bbb_t^x)$) and $Y_t=\pi_1(\bbb_t^y)$ (resp. $\pi_2(\bbb_t^y)$).
% 	We define $A_t$ the signed swept area, up to a constant, between the two paths $(X_s)_{s\leq t}$ and $(Y_s)_{s\leq t}$. We denote $\tilde{\mathcal{A}}_t\in ]-2\pi,2\pi]$ such that $\tilde{\mathcal{A}}_t\equiv A_t \mod(4\pi)$. We also denote by $R_t$ the distance between $X_t$ and $Y_t$.
% 	Then we have the following equivalence:
% 	\begin{equation}d_{cc}^2(\bbb_t^x,\bbb_t^{y})\sim R_t^2+|\tilde{\mathcal{A}}_t|.\end{equation}
% 	\end{prop}
	As a matter of fact, as in the Heisenberg group, a way to construct and study a coupling $(\bbb_t^x,\bbb_t^y)$ of Brownian motions on the subRiemannian manifold $SU(2)$ (resp. $SL(2,\mathbb{R})$), is to construct a coupling $(X_t,Y_t)$ of Brownian motions on the Riemannian manifold $S^2$ (resp. $\mathbb{H}^2$) and then to study the processes $R_t$ and $A_t$ described in Corollary \ref{distanceB2}. This is the method we will use in all this work. 
	%The processes $X_t$ and $Y_t$ will be called "driving noises".
	
	As explained before, one of the aim of this article is the determination of successful couplings in the model spaces. In fact, it is evident that such a coupling does not exist in the case of $SL(2,\mathbb{R})$ as we will explain just now. If $\left(\bbb_t,\bbb_t'\right)$ is a successful coupling on the subRiemannian manifold, then, $\left(\Pi_2(\bbb_t),\Pi_2(\bbb_t')\right)_t$ is a successful coupling of Brownian motions on the hyperbolic plane $\mathbf{H}^2$. It is well known that successful couplings do not exist on $\mathbf{H}^2$. As a proof, we can use Theorem (5.4) from Wang~\cite{wang2002liouville} (this result can be applied to every Riemannian manifold whose Ricci curvature is bounded below): as there exist some non constant bounded harmonic functions on the Riemannian manifold $\textbf{H}^2$, there is no successful coupling of Brownian motions on it. For $SU(2)$, we generalise the co-adapted successful coupling constructed by Kendall on the Heisenberg group~\cite{kendall2007coupling}. This is our main theorem:
		\begin{thm}\label{successful}
	There exists a co-adapted successful coupling in $SU(2)$.
	\end{thm}
	
	In the case of $\hh$, the coupling of Kendall consists in switching between two coupling models on the plane:
	\begin{itemize}
	    \item The first coupling is meant to act on the distance $R_t$ between the driving noises $X_t$ and $Y_t$. Kendall uses the reflection coupling so that $R_t$ hits $0$ at an a.s. finite time.
	    \item The second coupling is used to keep the swept area $A_t$ comparable to $R_t$. The idea is that, each time $|A_t|$ is too big in relation to $R_t$, $R_t$ stays constant and $|A_t|$ moves like a Brownian motion to reach an acceptable value. This is done with synchronous coupling, also called parallel transport coupling.
	\end{itemize}
	 The first difficulty to generalise this theorem to $SU(2)$ is to obtain explicit constructions of couplings on $S^2$ acting as described above. In particular, if the existence of couplings acting as expected on $R_t$ is known (see~\cite{ReflKuwada} for a reflection coupling and \cite{pascu2018couplings} for a coupling keeping $R_t$ constant) we need another way to describe them if we want to know their impact on the swept area. The second difficulty lies, of course, in the presence of the curvature (the driving noises $X_t$ and $Y_t$ live in $S^2$) that impacts some parts of the proof. The two possibilities to deal with this second difficulty will be to add a control on the upper bound of $R_t$ or to take the compactness of $SU(2)$ into account. Note that the first method enable us to successfully couple the Brownian motions on $S^2$ together with their signed swept area as values in $\mathbb{R}$ (see Theorem \ref{thm: couplingAreaInR}). This is a more general result than Theorem \ref{successful}.
	
To deal with the first difficulty, in this paper, we use Itô depiction of a Brownian motion in a frame in the sense of Emery (see~\cite{emery2012stochastic} for a basic introduction). Two Brownian motions $X_t$ and $Y_t$ on the Riemannian manifold $M=S^2$ can be described by the equations
\begin{equation}\label{equa}
	d^{\nabla}X_t=dU_1(t)e_1^X(t)+dU_2(t)e_2^X(t) \text{ and } d^{\nabla}Y_t=dV_1(t)e_1^Y(t)+dV_2(t)e_2^Y(t)
\end{equation} with:
\begin{itemize}
    \item $U(t):=(U_1(t),U_2(t))$ and $V(t):=(V_1(t),V_2(t))$ two Brownian motions in $\mathbb{R}^2$;
    \item $e^X(t):=(e_1^X(t), e_2^X(t))$ a continuous semi-martingale (adapted to the same filtration than $(X_t)_t$) in the orthonormal frame bundle $\mathcal{O}M$ above $(X_t)_t$.
    \item $e^Y(t):=(e_1^Y(t), e_2^Y(t))$ a continuous semi-martingale (adapted to the same filtration than $(Y_t)_t$) in the orthonormal frame bundle $\mathcal{O}M$ above $(Y_t)_t$.
\end{itemize}
Thus a coupling $(X_t,Y_t)$ is characterized by its starting points and the joint law of\\ $\left(\left(U(t),e^X(t)\right),\left(V(t),e^Y(t)\right)\right)_t$. By choosing a co-adapted coupling of \\$\left(\left(U(t),e^X(t)\right),\left(V(t),e^Y(t)\right)\right)_t$, we can describe and study a wide range of co-adapted couplings $(X_t,Y_t)$ on $M$.

In subsection \ref{subsec: Hessiennes}, we compute the first and second order derivatives of the signed swept area between two smooth curves in some specific basis of the tangent bundle (see Lemma \ref{HessA}). This allows us to obtain a general stochastic equation for $A_t$ in the co-adapted coupling described above. Using some well-known similar results about the distance in $S^2$, we also obtain a stochastic equation for $R_t$ (Lemma \ref{cov rho}).

In fact we can describe this structure for every simply connected Riemannian manifold $M$ with constant curvature $k$ and dimension $2$. Thus, even if no successful coupling can be constructed in $SL(2,\mathbb{R})$, we obtain a model of co-adapted couplings for our two model spaces with a way to compare the processes. All this is summed up in the following proposition:
		\begin{prop}\label{relations}
		Let $M$ be the simply connected Riemannian manifold of dimension $2$ with constant curvature $k$. We denote by $i(M)$ its injectivity radius and by $(x,y)\mapsto\rho(x,y)$ its usual Riemannian metric (note that $i(M)=+\infty$ if $k\leq 0$ and $i(M)=\frac{\pi}{\sqrt{k}}$ if $k<0$).\\ 
		Let us consider $X_0,Y_0\in M$ such that $0<\rho(X_0,Y_0)<i(M)$ and $\left(U(t),V(t)\right)_t$ a co-adapted coupling of Brownian motions on $\mathbb{R}^2$.\\
		For $t\in[0,T]$, we construct $(X_t)_t$ and $(Y_t)_t$ two processes on $M$ starting from $X_0$ and $Y_0$ with $T:=\inf\{t>0\ | \ \rho(X_t,Y_t)=0\}$ and such that $(X_t)_t$ and $(Y_t)_t$ satisfy (\ref{equa}) with:
			\begin{itemize}
			\item $e_1^X(t)=\frac{\exp_{X_t}^{-1}(Y_t)}{\rho(X_t,Y_t)}$ a unitary vector on $T_{X_t}M$;
			\item $e_2^X(t)$ such that $(e_1^X(t),e_2^X(t))$ is a positive orthonormal basis on $T_{X_t}M$;
			\item $(e_1^Y(t),e_2^Y(t))$ the parallel transport of $(e_1^X(t),e_2^X(t))$ along the geodesic joining $X_t$ and $Y_t$ (this defines a positive orthonormal basis on $T_{Y_t}M$).
		\end{itemize}
	 Then $(X_t,Y_t)_t$ is a co-adapted coupling of Brownian motions. Moreover, the processes $(R_t)_t$ and $(A_t)_t$ appearing in Corollary \ref{distanceB2} satisfy:
		\begin{align}
			dR_t&=dV_1(t)-dU_1(t)+\sqrt{k}\cot(\sqrt{k}R_t)dt-\frac{\sqrt{k}}{\sin(\sqrt{k}R_t)} dU_2(t)\cdot dV_2(t)\label{eq: R(t)}\\
			dA_t&=\frac{\tan(\frac{\sqrt{k}R_t}{2})}{\sqrt{k}}(dU_2(t)+dV_2(t))+\frac{1}{2\cos^2(\frac{\sqrt{k}R_t}{2})}( dU_2\cdot dV_1(t)- dV_2(t)\cdot dU_1(t))\label{eq: A(t)}\\
			dR_t\cdot dA_t&=\frac{1}{\sqrt{k}}\tan\left(\frac{\sqrt{k}R_t}{2}\right)\left( dV_1(t)\cdot dU_2(t)- dU_1(t)\cdot dV_2(t)\right)
		\end{align}
		with $dU_i(t)\cdot dV_j(t)$ denoting the derivative of the joint quadratic variation of $U_i$ and $V_j$.
	\end{prop}
	Note that, if $k$ tends to $0$, we obtain some well known relations for the Heisenberg group.
	Thanks to this formula, we are able to give some examples, including reflection coupling as well as a coupling keeping $R_t$ constant. Note also that this last coupling is not the classical parallel transport one but needs the introduction of some additive noise to compensates the curvature. Finally, under some conditions, Proposition \ref{relations} can be true even if $R_t\in\{0,i(M)\}$ as explained briefly in Remark \ref{distanceNulle}. 
	
\subsection{Organisation of the paper}\label{subsec: Organisation}
The structure of the paper is as follows. In section \ref{sec: Preliminaries} we give some preliminaries: some generalities about subRiemannian structure, subLaplacian operator, a presentation of the model spaces and of the cylindrical coordinates as well as the representation of the Brownian motions in cylindrical coordinates on the model spaces. Note that an alternative proof for the expression of the subLaplacian operator in cylindrical coordinates can be found in Appendices \ref{SubLapSU(2)} and \ref{SubLapSL(2)}. Section \ref{sec: Comparison} contains the proofs for Proposition \ref{prop: distance ponctuelle} and Corollary \ref{distanceB2}. In Section \ref{sec: Coupling on manifolds} we compute the first and second order derivatives of the signed swept area. We also give the proof of Proposition \ref{relations} and important examples of co-adapted couplings. Finally, in Section \ref{sec: Successful coupling}, we give the construction of the announced successful coupling and thus the proof of Theorem \ref{successful}. 
  \section{Preliminaries}\label{sec: Preliminaries}
  Note that most of the content of these preliminaries can be found in \cite{baudoin2009subelliptic,bonnefont2012subelliptic,bonnefont-these}. 
  \subsection{The subRiemannian structure}\label{subsec: Structure}
  	Let $G$ be a smooth connected Riemannian manifold of dimension $N$ and $n\leq N$ an integer. For each $x\in G$, we define a vector subspace $\mathcal{H}_x$ of dimension $n$ of $T_xG$, the tangent space in $x$ . This way we define a subbundle of the tangent bundle $TG$ denoted by $\mathcal{H}$ and called horizontal bundle. We can then define the horizontal curves, that is, the smooth curves $\gamma:I\subset \mathbb{R}\to G$ such that $\dot{\gamma}(t)\in \mathcal{H}_{\gamma(t)}$: the horizontal curves are "moving" only with directions in $\mathcal{H}$. Defining a scalar product $\langle\cdot,\cdot\rangle_{\mathcal{H}_{x}}$ on $\mathcal{H}_{x}$ for all $x\in G$, we obtain the length $L(\gamma)$ of the horizontal curve $\gamma$:
	\begin{equation*}
		L(\gamma):=\int_I \sqrt{\langle\dot{\gamma}(t),\dot{\gamma}(t)\rangle_{\mathcal{H}_{\gamma(t)}}}dt.
	\end{equation*}
The Carnot-Carathéodory distance between $x$ and $y\in G$ is finally defined by:
	\begin{equation*}
		d_{cc}(x,y)=\inf\{ L(\gamma) \ | \ \gamma \text{ horizontal curve between } x \text{ and }y\}.
	\end{equation*}
	Let us suppose that the horizontal space satisfies the Hörmander condition, that is, provided $(X_1,...,X_n)$ is a local basis of vector fields in $\mathcal{H}$, the tangent space $TG$ is spanned by the vectors along with all their commutators (obtained by operation with Lie brackets). Then the Carnot-Carathéodory distance is finite and the subRiemannian structure is well defined.
	In the case of stratified Lie groups, we can choose some globally defined and left-invariant vector fields that we will denote $\bar{X}_1,...,\bar{X}_n$. We can then introduce the subLaplacian operator: \begin{equation*}L:=\frac{1}{2}\sum_{1}^{n}\bar{X_i}^2.\end{equation*}
	Because Hörmander conditions are satisfied, this operator is hypoelliptic. This is this operator that we will use to define our Brownian motions.
	
	\subsection{Definition of {$SU(2)$} and {$SL(2,\mathbb{R})$} and cylindrical coordinates}\label{subsec: Definition}
	We begin with the presentation of the two space models and the cylindrical coordinates.
	\begin{itemize}
		\item By $SU(2)$, we denote the group of the unitary two dimensional matrices with complex coefficients and with determinant 1. Considering the manifold structure induced of the usual topology on the matrices group, this is a Lie group.
		%for the law induced by the multiplication of matrices. 
%Considering the manifold structure induced by the usual topology on the matrices group and, as the application $\Big{\{}\begin{array}{ccc}
% 			SU(2)\times SU(2)&\to &SU(2)\\
% 			(A,B)&\mapsto &A^{-1}\cdot B
% 		\end{array}$ is smooth, $SU(2)$ is a Lie group. 
		Note that we have: \begin{align*}
		SU(2)&=\Big{\{}
		\begin{pmatrix}
			z_1 & z_2\\
			-\bar{z_2} & \bar{z_1}
		\end{pmatrix}
		,\ z_1,\ z_2\in\mathbb{C},\ |z_1|^2+|z_2|^2=1 \Big{\}}\\
		&=\Big{\{}
		\begin{pmatrix}
			\cos(\eta)e^{i\theta_1} & \sin(\eta)e^{i\theta_2}\\
			-\sin(\eta)e^{-i\theta_2} & \cos(\eta)e^{-i\theta_1}
		\end{pmatrix}
		,\ \eta\in\left[0,\frac{\pi}{2}\right],\ \theta_1,\ \theta_2\in[0,2\pi] \Big{\}}.\end{align*}
		The associated Lie algebra is $\mathfrak{su}(2)=\{M\in M_{2}(\mathbb{C}),\ \exp(tM)\in SU(2) \ \forall t>0\}$. It is constituted by the skew-adjoint two dimensional matrices with complex coefficients and trace 0. It is also the tangent space of $SU(2)$ at point $I_2$. A basis of this algebra can be formed by the Pauli matrices. We will use these Pauli matrices up to the multiplicative coefficient $\frac{1}{2}$, we denote: \begin{equation*}X=\frac{1}{2}\begin{pmatrix} 0 &1 \\ -1 &0 \end{pmatrix}, Y=\frac{1}{2}\begin{pmatrix} 0 &i \\ i &0 \end{pmatrix} \text{ and }Z=\frac{1}{2}\begin{pmatrix} i &0 \\ 0 &-i \end{pmatrix}.\end{equation*}
		Then $(X,Y,Z)$ is a basis of $\mathfrak{su(2)}$ and, thanks to the multiplicative coefficient, it also satisfies:
		\begin{equation}\label{LieSU} [X,Y]=Z\ , \ [Y,Z]=X\  \text{and } [Z,X]=Y.\end{equation}
		%This last property will be useful in the following sections.
		\\
		It is important to notice that all matrices in $SU(2)$ can be written on the form:
		\begin{equation*}\exp(\ph(\cos(\theta)X+\sin(\theta)Y))\exp(zZ)=\begin{pmatrix}
			\cos\left(\frac{\ph}{2}\right)e^{i\frac{z}{2}} & e^{i(\theta-\frac{z}{2})}\sin\left(\frac{\ph}{2}\right)\\[5 pt]
			-e^{-i(\theta-\frac{z}{2})}\sin\left(\frac{\ph}{2}\right) & \cos\left(\frac{\ph}{2}\right)e^{-i\frac{z}{2}} 
		\end{pmatrix}\end{equation*} 
		(this result is trivial taking $\ph=2\eta$, $z=2\theta_1$ and $\theta\equiv\theta_2+\theta_1\mod(2\pi)$). Thus, we have described a system of coordinates $(\ph,\theta,z)$ for $SU(2)$ with $\ph\in[0,\pi]$, $z\in]-2\pi,2\pi]$ and $\theta\in[0,2\pi[$ called the cylindrical coordinates. We can also consider the coordinate system induced by $\exp(xX+yY)\exp(zZ)$ with $(x,y)\in \mathbb{R}^2$, $z\in]-2\pi,2\pi]$.\\
		Let us remark that the cylindrical coordinates are a good way to observe the link between the sphere $S^2$ and $SU(2)$. Indeed, as there is a trivial diffeormophism between $SU(2)$ and $S^3$, using the Hopf fibration, we can define a submersion from $SU(2)$ to $S^2$. For example, we can define it using quaternions. For $\begin{pmatrix}
			z_1 & z_2\\
			-\bar{z_2} & \bar{z_1}
		\end{pmatrix}\in SU(2)$, with $z_1=x_1+iy_1$ and $z_2=x_2+iy_2$, we define a unique quarternion $q=x_1+x_2i+y_2j+y_1k$. Denoting $N_0=k$ the north pole in $S^3$, we define:
		\begin{equation*}
		\begin{array}{cccc}
			\Pi_1:SU(2)&\to &S^2&\\
			q&\mapsto &qkq^*&=2(x_1y_2+x_2y_1)i+2(y_1y_2-x_1x_2)j+(x_1^2-x_2^2-y_2^2+y_1^2)k
		\end{array}.
		\end{equation*}
		 We can show~(\cite{bonnefont-these}) that $\Pi_1$ define a submersion. Using the cylindrical coordinates $(\ph,\theta,z)$ for $q$, we obtain:\begin{equation*} \Pi_1(q)=\sin(\ph)\sin(\theta)i-\sin(\ph)\cos(\theta)j+\cos(\ph)k.\end{equation*} Thus $\Pi_1$ sends every element of $SU(2)$ described by the cylindrical coordinates $(\ph,\theta,z)$ onto the point of $S^2$ described by the spherical coordinates $(\ph,\theta)$. Moreover the fiber over $(\ph,\theta)$ of this projection is described by $\{(\ph,\theta,z),\ z\in]-2\pi,2\pi]\}$. 
		\item We now deal with the $SL(2,\mathbb{R})$ group. It is the group of two dimensional matrices with real coefficients and with determinant $1$. As for $SU(2)$ this is a Lie group with the topology of the matrices groups. The associated Lie algebra, denoted $\mathfrak{sl}(2)$, is  constituted by the two dimensional matrices with real coefficients and trace 0. It is also the tangent space of $SL(2,\mathbb{R})$ at point $I_2$. The following matrices, using the same notation as for $SU(2)$, form a basis of $\mathfrak{sl}(2)$:   $X=\frac{1}{2}\begin{pmatrix} 1 &0 \\ 0 &-1 \end{pmatrix}$, $Y=\frac{1}{2}\begin{pmatrix} 0 &-1 \\ -1 &0 \end{pmatrix}$ and $Z=\frac{1}{2}\begin{pmatrix} 0 &-1 \\ 1 &0 \end{pmatrix}$. This time, the relation induced by the Lie brackets are:
		\begin{equation}\label{LieSL} [X,Y]=Z\ , \ [Y,Z]=-X\  \text{and } [Z,X]=-Y.\end{equation}
		As before, we can prove that every element of $SL(2,\mathbb{R})$ can be written on the form :
		\begin{align*}&\exp(\ph(\cos(\theta)X+\sin(\theta)Y))\exp(zZ)=\\
		&\begin{pmatrix}
			\cosh\left(\frac{\ph}{ 2}\right)\cos\left(\frac{ z}{2}\right)+\sinh\left(\frac{\ph}{ 2}\right)\cos\left(\theta+\frac{ z}{2}\right) &-\cosh\left(\frac{\ph}{2}\right)\sin\left(\frac{z}{2}\right)-\sinh\left(\frac{\ph}{2}\right)\sin\left(\theta+\frac{z}{2}\right) \\[6 pt]
			\cosh\left(\frac{\ph}{2}\right)\sin\left(\frac{z}{2}\right)-\sinh\left(\frac{\ph}{2}\right)\sin\left(\theta+\frac{z}{2}\right) & \cosh\left(\frac{ \ph}{ 2}\right)\cos\left(\frac{ z}{2}\right)-\sinh\left(\frac{\ph}{ 2}\right)\cos\left(\theta+\frac{ z}{2} \right)
		\end{pmatrix}
\end{align*} 
		with $(x,y)\in \mathbb{R}^2$, $z\in]-2\pi,2\pi]$, $\ph>0$ and $\theta\in[0,2\pi[$  as seen in~\cite{Boscain} for example. We can thus define as well a system of cylindrical coordinates on $SL(2,\mathbb{R})$.\\
	Similarly to $SU(2)$, we can define a submersion $\Pi_2$ from $SL(2,\mathbb{R})$ to $\textbf{H}^2$, the Poincaré upper half-plane: 
		\begin{equation*}\begin{array}{ccc}
			\Pi_2:SL(2)&\to &\textbf{H}^2\\[5pt]
			M=\begin{pmatrix}
			a & b\\
			c & d
		\end{pmatrix}&\mapsto & \dfrac{ai+b}{ci+d}
		\end{array}.
		\end{equation*}
		 Using the cylindrical coordinates $(\ph,\theta,z)$ for $M$, we get: \begin{equation*}\Pi_2(M)=\dfrac{i-\sinh(\ph)\sin(\theta)}{\cosh(\ph)-\sinh(\ph)\cos(\theta)}.\end{equation*} With the help of the Cartesian formula of the hyperbolic metric and trigonometric relations, we obtain that $\Pi_2\left(\ph,\theta,z\right)$ is described by the polar coordinates $(\ph,\theta)$ relative to the pole $N_0=i$. Thus the fiber over $(\ph,\theta)$ of this projection is given by $\{(\ph,\theta,z),\ z\in]-2\pi,2\pi]\}$. 
	\end{itemize}
	For the two cases we denote $\bar{X}$, $\bar{Y}$ and $\bar{Z}$ the left-invariant vector fields associated to $X$, $Y$, $Z$, that is, the vector fields induced by: \begin{equation*}\frac{\partial}{\partial\eps}_{|\eps=0}\left(\exp\big(\ph(\cos(\theta)X+\sin(\theta)Y)\big)\exp(zZ)\exp(\eps A)\right)\text{ for }A=X,Y,Z.\end{equation*}
	We can then provide a subRiemannian structure to $SU(2)$ (respectively $SL(2,\mathbb{R})$) by considering the horizontal bundle $\mathcal{H}=\Span\langle\bar{X},\bar{Y}\rangle$ and the associated Carnot-Carathéodory distance $d_{cc}$. It has been proven by Baudoin and Bonnefont in~\cite{baudoin2009subelliptic,bonnefont2012subelliptic,bonnefont-these}, that there exist two constants $c_1$ and $c_2$ such that, for all elements of $SU(2)$ (resp. $SL(2,\mathbb{R})$) written in cylindrical coordinates $(\ph,\theta,z)$:
	\begin{equation}\label{equiv} c_1(\ph^2+|z|)\leq d_{cc}^2(0,(\ph,\theta,z))\leq c_2 (\ph^2+|z|).
	\end{equation}
	We now interest ourselves in the subLaplacian operator. Here it is given by $L=\frac{1}{2}\left(\bar{X}^2+\bar{Y}^2\right)$. Note that, as the considered vector fields are chosen left invariant, the subRiemannian structure is in fact only determined by the horizontal bundle at point $I_2$, $\mathcal{H}_{I_2}=\Span\langle X,Y\rangle$. In cylindrical coordinates, we have the following expression for the subLaplacian operator:
	\begin{align}
	\text{for $SU(2)$:}	&L=\frac{1}{2}\left(\partial^2_{\ph,\ph}+\frac{1}{\sin^2(\ph)}\partial^2_{\theta,\theta}+\tan^2\left(\frac{\ph}{2}\right)\partial^2_{z,z}+\frac{1}{\cos^2\left(\frac{\ph}{2}\right)}\partial^2_{\theta,z}+\cot(\ph)\partial_{\ph}\right);\label{SubLapfSU}\\
	\text{for $SL(2,\mathbb{R})$:}	&L=\frac{1}{2}\left(\partial^2_{\ph,\ph}+\frac{1}{\sinh^2(\ph)}\partial^2_{\theta,\theta}+\tanh^2\left(\frac{\ph}{2}\right)\partial^2_{z,z}+\frac{1}{\cosh^2\left(\frac{\ph}{2}\right)}\partial^2_{\theta,z}+\coth(\ph)\partial_{\ph}\right)\label{SubLapfSL}.
	\end{align}
		In the literature, these results are obtained using direct matrix computations as in~\cite{bonnefont-these} or (only for $SU(2)$) in~\cite{Ruzhansky-Turunen}\footnote{Here the author uses Euler angle parametrization which is quite similar from the one with cylindrical coordinates.} . We give a different way to prove this result in \ref{SubLapSU(2)} and \ref{SubLapSL(2)} only using the Lie brackets relations (\ref{LieSU}) and (\ref{LieSL}). 
	\subsection{Geometrical interpretation of the Brownian motion}\label{subsec: Interpretation}
The following Proposition gives a formal geometrical interpretation of the Brownian motion on $SU(2)$ and $SL(2,\mathbb{R})$:
\begin{prop}[\cite{baudoin2009subelliptic,bonnefont2012subelliptic,bonnefont-these}]\label{distanceB}
	Let us consider a Brownian motion $\bbb_t$ on $SU(2)$ (resp. $SL(2,\mathbb{R})$) and let $(\ph_t,\theta_t,z_t)$ be its cylindrical coordinates. We have:
\begin{itemize}
    \item The process $\left(X_t:=\Pi_1(\bbb_t)\right)_t$ (resp. $\Pi_2(\bbb_t)$) is a Brownian motion on the sphere $S^2$ (resp. the hyperbolic plane $\mathbf{H}^2$). Its spherical coordinates are given by $(\ph_t,\theta_t)$ with respect to the pole $N_0$ fixed by the submersion.
    \item $\zeta_t-\zeta_0$ is the signed swept area (modulo $4\pi$) of $X_t$ with respect to the fixed pole $N_0$.
\end{itemize}
\end{prop}
We briefly explain the above results. More details can be found in~\cite{baudoin2009subelliptic,bonnefont2012subelliptic,bonnefont-these}.
We define the Brownian motion on $SU(2)$ (resp. $SL(2,\mathbb{R})$) as the diffusion process $(\bbb_t)_t$ with infinitesimal generator $L$. Using cylindrical coordinates, there exist $(\ph_t)_t$, $(\theta_t)_t$ and $(z_t)_t$ three real diffusion processes, such that:  
\begin{equation*}\bbb_t=\exp(\ph_t(\cos(\theta_t)\bar{X}+\sin(\theta_t)\bar{Y}))\exp(z_t \bar{Z}) \text{ and }  
	\begin{cases}
			d\ph_t= dB_t^1+\frac{1}{2}\sqrt{k}\cot(\sqrt{k}\ph_t)dt\\
			d\theta_t=\frac{\sqrt{k}}{\sin(\sqrt{k}\ph_t)}dB_t^2\\
			dz_t=\frac{\tan\left(\frac{\sqrt{k}\ph_t}{2}\right)}{\sqrt{k}}dB_t^2
		\end{cases}
	\end{equation*}
with $B_t^1$ and $B_t^2$ two independent Brownian motions in $\mathbb{R}$. In the above formula, we take $k=1$ for the case of $SU(2)$ and $k=-1$ for the case of $SL(2,\mathbb{R})$.
	Using the first two equations, we get the infinitesimal generator of the diffusion $(\ph_t,\theta_t)$: \begin{equation*}\frac{1}{2}\left(\partial^2_{\ph,\ph}+\frac{k}{\sin^2(\sqrt{k}\ph)}\partial^2_{\theta,\theta}+\sqrt{k}\cot(\sqrt{k}\ph)\partial_{\ph}\right).\end{equation*} 
This is the Laplace Beltrami operator on the sphere $S^2$ (resp. the hyperbolic plane $\textbf{H}^2$) in spherical (resp. polar) coordinates.
	Moreover the signed area $\int_0^t(\cos(\sqrt{k}\ph_s)-1)d\theta_s$ swept by the path of $(\ph_t,\theta_t)_t$ on $S^2$ (resp. $\textbf{H}^2$) with respect of the pole $N_0$, is exactly equal to $z_t-z_0$.
\begin{NB}
	Note that this swept area is a signed value depending of the orientation of the curve $(\ph_t,\theta_t)_t$. For example, if $\theta_t$ is decreasing for all $t$, then we get $z_t-z_0<0$.
	\end{NB}
	
	\section{Comparison of two elements in {$SU(2)$} and {$SL(2,\mathbb{R})$}}\label{sec: Comparison}
%	\subsection{Geometrical interpretation of the distance}\label{subsec: InterpretationDistanceSU}
	Let us now deal with the proof of Proposition \ref{prop: distance ponctuelle} in the case of $SU(2)$. The case of $SL(2,\mathbb{R})$ uses similar computations and can be found in Appendix \ref{subsec: InterpretationDistanceSL}.
	\begin{proof}[Proof of Proposition (\ref{prop: distance ponctuelle}) for $SU(2)$]
	 Let $g=(\ph,\theta,z)$ and $g'=(\ph',\theta',z')$ be two elements of $SU(2)$ written in cylindrical coordinates. We denote by $(\rho,\Theta,\zeta)$ the cylindrical coordinates of $g^{-1}\cdot g'$. Using matrix computations, we have: \begin{align*}g^{-1}\cdot g'=&\left(\exp\big(\ph(\cos(\theta)X+\sin(\theta)Y)\big)\exp(zZ)\right)^{-1}\exp\big(\ph'(\cos(\theta')X+\sin(\theta')Y)\big)\exp(z'Z)\\
%\end{align*}
% 	Considering the matrix form of $X$ and $Y$, we get: \begin{align*}\exp(\ph(\cos(\theta)X+\sin(\theta)Y))^{-1}&=\begin{pmatrix}
% 			\cos\left(\frac{\ph}{2}\right) & e^{i\theta}\sin\left(\frac{\ph}{2}\right)\\
% 			-e^{-i\theta}\sin\left(\frac{\ph}{2}\right) & \cos\left(\frac{\ph}{2}\right)
% 		\end{pmatrix}^{-1}\\
% 		&=\begin{pmatrix}
% 			\cos\left(\frac{\ph}{2}\right) & -e^{i\theta}\sin\left(\frac{\ph}{2}\right)\\
% 			e^{-i\theta}\sin\left(\frac{\ph}{2}\right) & \cos\left(\frac{\ph}{2}\right)
% 		\end{pmatrix}\\
% 		&=\exp(-\ph(\cos(\theta)X+\sin(\theta)Y)).\end{align*} 
% 		We also have: $exp(\zetaZ)^{-1}=exp(-\zetaZ)$.
% 		For $\ph\in[0,\pi]$, $\theta\in[0,2\pi[ $ and $z\in]-2\pi,2\pi]$, we get  \begin{equation*}
% 		\exp(\zetaZ)\exp(\ph(\cos(\theta)X+\sin(\theta)Y))=\exp\big{(}\ph(\cos(z+\theta)X+\sin(\theta-z)Y)\big{)}\exp(-\zetaZ).
% 	\end{equation*}
% 	Thus $x^{-1}y$ can be written:
% 	\begin{align*}
% 		\big(\exp(&-z Z)\exp\big(-\ph(\cos(\theta)X+\sin(\theta)Y)\big)\big)\left(\exp\big(\ph'(\cos(\theta')X+\sin(\theta')Y)\big)\exp(z' Z)\right)\\
% 		=&\Big(\exp\big(-\ph(\cos(\theta-z)X+\sin(\theta-z)Y)\big)\exp(-zZ)\Big)\\
% 		&\Big(\exp\big(\ph'(\cos(\theta')X+\sin(\theta')Y)\big)\exp(z'Z)\Big)\\
		=&\exp\big(-\ph(\cos(\theta-z)X+\sin(\theta-z)Y)\big)\\
		&\exp\big(\ph'(\cos(\theta'-z)X+\sin(\theta'-z)Y)\big)\exp\big((z'-z)Z\big)\\
% 		=&\exp\big(\rho(\cos(\theta)X+\sin(\theta)Y)\big)\exp(\zetaZ)\exp\big((z'-z)Z\big)
 	\end{align*}
 	Then $\exp\big(\rho(\cos(\Theta)X+\sin(\Theta)Y)\big)\exp(\zeta Z)$ is equal to the matrix product:
	\begin{equation*}
	\begin{pmatrix}
			\cos\left(-\frac{\ph}{2}\right) & e^{i\left(\theta-z\right)}\sin\left(-\frac{\ph}{2}\right)\\[6pt]
			-e^{-i\left(\theta-z\right)}\sin\left(-\frac{\ph}{2}\right) & \cos\left(-\frac{\ph}{2}\right)
		\end{pmatrix}\begin{pmatrix}
			\cos\left(\frac{\ph'}{2}\right) & e^{i(\theta'-z)}\sin\left(\frac{\ph'}{2}\right)\\[6pt]
			-e^{-i(\theta'-z)}\sin\left(\frac{\ph'}{2}\right) & \cos\left(\frac{\ph'}{2}\right)
		\end{pmatrix}.
	\end{equation*}
	In particular we have:\begin{equation}\label{S1}%\begin{subnumcases}{}
		\cos\left(\frac{\rho}{2}\right)e^{i\frac{\zeta}{2}}=
		\cos\left(\frac{\ph}{2}\right)\cos\left(\frac{\ph'}{2}\right)+e^{i(\theta-\theta')}\sin\left(\frac{\ph}{2}\right)\sin\left(\frac{\ph'}{2}\right).
		%\label{S1} \\
		%e^{i\left(\Theta-\frac{\zeta}{2}\right)}\sin\left(\frac{\rho}{2}\right)= \cos\left(\frac{\ph}{2}\right)\sin\left(\frac{\ph'}{2}\right)e^{i(\theta'-z)}-\sin\left(\frac{\ph}{2}\right)\cos\left(\frac{\ph'}{2}\right)e^{i(\theta-z)}\label{S2}.
	%\end{subnumcases}
	\end{equation}
	From this equation we can get the expected geometrical interpretation on the values $\rho$ and $\zeta$.
	\begin{itemize}
		\item We begin with $\rho$: we are going to prove that $\rho$ is the usual distance on $S^2$ between $x$ and $y$ with $x:=\Pi_1(g)$ and $y:=\Pi_1(g')$. Let us denote $\rho(x,y)$ this distance.\\
		We note that:
		\begin{align*}
			\cos(\rho(x,y))&=<x,y>_{\mathbb{R}^3}\\
			&=\sin(\theta)\sin(\theta')\sin(\ph)\sin(\ph')+\cos(\theta)\cos(\theta')\sin(\ph)\sin(\ph')\\
			&+\cos(\ph)\cos(\ph')\\
			&=\cos(\ph)\cos(\ph')+\sin(\ph)\sin(\ph')\cos(\theta-\theta').
		\end{align*}
		From (\ref{S1}), we get \begin{align*}
			\cos^2&\left(\frac{\rho}{2}\right)=
			\left(\cos\left(\frac{\ph}{2}\right)\cos\left(\frac{\ph'}{2}\right)+\cos\left(\theta-\theta'\right)\sin\left(\frac{\ph}{2}\right)\sin\left(\frac{\ph'}{2}\right)\right)^2\\
			&+\sin^2(\theta-\theta')\sin^2\left(\frac{\ph}{2}\right)\sin^2\left(\frac{\ph'}{2}\right)\\
%			&=
%			\cos^2\left(\frac{\ph}{2}\right)\cos^2\left(\frac{\ph'}{2}\right)+\sin^2\left(\frac{\ph}{2}\right)\sin^2\left(\frac{\ph'}{2}\right)\\
%			&+2\cos\left(\frac{\ph}{2}\right)\cos\left(\frac{\ph'}{2}\right)\cos(\theta-\theta')\sin\left(\frac{\ph}{2}\right)\sin\left(\frac{\ph'}{2}\right)\\
%			&=\frac{(1+\cos(\ph))(1+\cos(\ph'))}{4}+\frac{(1-\cos(\ph))(1-\cos(\ph'))}{4}\\
%			&+\frac{\sin(\ph)\sin(\ph')
%				\cos(\theta-\theta')}{2}\\
			&=\frac{1+\cos(\ph)\cos(\ph')}{2}+\frac{\sin(\ph)\sin(\ph')
				\cos(\theta-\theta')}{2}\\
			&=\frac{1+\cos(\rho(x,y))}{2}=\cos^2\left(\frac{\rho(x,y)}{2}\right).
		\end{align*}
		We obtain the announced result as $\rho$ and $\rho(\Pi_1(x),\Pi_1(y))$ live in $[0,\pi]$.\\ Moreover, we get that \begin{equation}\label{Rd}
			\cos(\rho)=\cos(\ph)\cos(\ph')+\sin(\ph)\sin(\ph')\cos(\theta-\theta').
		\end{equation}
		\item We now deal with $\zeta$. We consider $\mathcal{A}_{y,x,N_0}$ the area of the spherical triangle with vertices $y$, $x$ and $N_0$, the north pole. The lengths of the opposite sides are respectively $\ph$, $\ph'$ and $\rho$.
		% Now, with the formula of L'Huillier, we obtain:
		%\begin{align*}
		%	\tan^2(\frac{\mathcal{A}_{y,x,N_0}}{4})&=\tan(\frac{R+R+R'}{4})\tan(\frac{R+R'-R}{4})\tan(\frac{R+R'-R}{4})\tan(\frac{R+R-R'}{4})\\
		%	&=\frac{\sin(\frac{R+R+R'}{4})\sin(\frac{R+R'-R}{4})\sin(\frac{R+R'-R}{4})\sin(\frac{R+R-R'}{4})}{\cos(\frac{R+R+R'}{4})\cos(\frac{R+R'-R}{4})\cos(\frac{R+R'-R}{4})\cos(\frac{R+R-R'}{4})}\\
		%	&=\frac{(\cos(\frac{R}{2})-\cos(\frac{R+R'}{2}))(\cos(\frac{R-R'}{2})-\cos(\frac{R}{2}))}{(\cos(\frac{R}{2})+\cos(\frac{R+R'}{2}))(\cos(\frac{R-R'}{2})+\cos(\frac{R}{2}))}\\	&=-\frac{\cos^2(\frac{R}{2})-\cos(\frac{R}{2})(\cos(\frac{R+R'}{2})+\cos(\frac{R-R'}{2}))+\cos(\frac{R+R'}{2})\cos(\frac{R-R'}{2})}{\cos^2(\frac{R}{2})+\cos(\frac{R}{2})(\cos(\frac{R+R'}{2})+\cos(\frac{R-R'}{2}))+\cos(\frac{R+R'}{2})\cos(\frac{R-R'}{2})}\\
		%	&=-\frac{\frac{1}{2}(1+\cos(R))-2\cos(\frac{R}{2})\cos(\frac{R}{2})\cos(\frac{R'}{2})+\frac{1}{2}(\cos(R)+\cos(R'))}{\frac{1}{2}(1+\cos(R))+2\cos(\frac{R}{2})\cos(\frac{R}{2})\cos(\frac{R'}{2})+\frac{1}{2}(\cos(R)+\cos(R'))}\\
		%	&=\frac{4\cos(\frac{R}{2})\cos(\frac{R}{2})\cos(\frac{R'}{2})-(1+\cos(R)+\cos(R)+\cos(R'))}{4\cos(\frac{R}{2})\cos(\frac{R}{2})\cos(\frac{R'}{2})+(1+\cos(R)+\cos(R)+\cos(R'))}.\\
		%\end{align*}
		%On the other side, we get: $\tan^2(\frac{z}{4})=\frac{1-\cos(\frac{z}{2})}{1+\cos(\frac{z}{2})}$. 
		Using an equivalent of the Heron formula for the spherical triangle (see~\cite{janson2015euclidean}), we have: \begin{equation*}\cos\left(\frac{\mathcal{A}_{y,x,N_0}}{2}\right)=\frac{1}{4\cos\left(\frac{\rho}{2}\right)\cos\left(\frac{\ph}{2}\right)\cos\left(\frac{\ph'}{2}\right)}\left(1+\cos(\ph)+\cos(\ph')+\cos(\rho)\right).\end{equation*}
		Because of (\ref{S1}) and (\ref{Rd}), we get:
		\begin{align*}
			&\cos\left(\frac{\zeta}{2}\right)\\
%			&=\frac{1}{\cos\left(\frac{\rho}{2}\right)}\left(\cos\left(\frac{\ph}{2}\right)\cos\left(\frac{\ph'}{2}\right)+\cos\left(\theta-\theta'\right)\sin\left(\frac{\ph}{2}\right)\sin\left(\frac{\ph'}{2}\right)\right)\\
			&=\frac{1}{\cos\left(\frac{\rho}{2}\right)}\left(\cos\left(\frac{\ph}{2}\right)\cos\left(\frac{\ph'}{2}\right)+\frac{\cos(\rho)-\cos(\ph)\cos(\ph')}{\sin(\ph)\sin(\ph')}\sin\left(\frac{\ph}{2}\right)\sin\left(\frac{\ph'}{2}\right)\right)\\
%			&=\frac{1}{4\cos\left(\frac{\rho}{2}\right)\cos\left(\frac{\ph}{2}\right)\cos\left(\frac{\ph'}{2}\right)}\left(4\cos^2\left(\frac{\ph}{2}\right)\cos^2\left(\frac{\ph'}{2}\right)+\cos(\rho)-\cos(\ph)\cos(\ph')\right)\\
			&=\frac{1}{4\cos\left(\frac{\rho}{2}\right)\cos\left(\frac{\ph}{2}\right)\cos\left(\frac{\ph'}{2}\right)}\left(\left(1+\cos(\ph)\right)\left(1+\cos(\ph')\right)+\cos(\rho)-\cos(\ph)\cos(\ph')\right)\\
%			&=\frac{1}{4\cos\left(\frac{\rho}{2}\right)\cos\left(\frac{\ph}{2}\right)\cos\left(\frac{\ph'}{2}\right)}\left(1+\cos(\ph)+\cos(\ph')+\cos(\rho)\right).
&=\cos\left(\frac{\mathcal{A}_{y,x,N_0}}{2}\right)
		\end{align*}
		Thus we have $|\zeta|\equiv  \mathcal{A}_{y,x,N_0}\mod(4\pi)$.\\
		Moreover, still using (\ref{S1}), we get that $\sin\left(\frac{\zeta}{2}\right)=\sin(\theta-\theta')\frac{\sin\left(\frac{\ph}{2}\right)\sin\left(\frac{\ph'}{2}\right)}{\cos\left(\frac{\rho}{2}\right)}$, and so $\zeta<0$ if and only if $\theta<\theta'$. This gives the expected result.
	\end{itemize}
	As the Carnot-Carathéodory distance is left invariant, we have $d_{cc}(g,g')=d_{cc}(0,g^{-1}\cdot g')$ and we directly deduce (\ref{eq: equivPonctuelle}) from (\ref{equiv}).
	\end{proof}
	Applying Proposition \ref{prop: distance ponctuelle}, we get $\zeta_t\equiv z_t'-z_t\mod (4\pi)+\textrm{sign}(\theta_t-\theta_t') \mathcal{A}_{Y_t,X_t,N_0}$. To prove Corollary \ref{distanceB2}, we just need to remark that $z_t'+\textrm{sign}(\theta_t-\theta_t') \mathcal{A}_{Y_t,X_t,N_0}-z_t-\textrm{sign}(\theta_0-\theta_0') \mathcal{A}_{Y_0,X_0,N_0}$ is equal to the continuous process $A_t$ announced in Corollary \ref{distanceB2} that can be defined as follow:
    \begin{defi}\label{def: aireBalayee1}
        Let $(X_t)_t$ and $(Y_t)_t$ be two Brownian motions on $S^2$ (resp. $\mathbf{H}^2$) such that $Y_s$ is not in the cut locus of $X_s$ for all $s\leq t$. We define the signed swept area $A_t$ as the signed area of the oriented loop starting from $Y_0$, following:	
        \begin{itemize}
	    \item the path ${s\mapsto Y_s, s\leq t}$ from $Y_0$ to $Y_t$;
	    \item the geodesic joining $Y_t$ to $X_t$;
	    \item the path ${s\mapsto X_{t-s}, s\leq t}$ joining $X_t$ to $X_0$;
	    \item finally the geodesic joining $X_0$ to $Y_0$.
	    \end{itemize}
    \end{defi}
    Note that this definition makes sense for any oriented Riemannian manifold. Here the sign of the quantity $A_t$ changes when the paths of $Y_t$ and $X_t$ are crossing. Note also that $A_t$ takes values in $\mathbb{R}$ whereas $\zeta_t$ takes values in $]-2\pi,2\pi]$.
    
% 	\begin{proof}[Proof of Corollary \ref{distanceB2} for $SU(2)$]\label{Def1 Swept area}
% We just need to give a geometric interpretation of the quantity $z_t'-z_t+\textrm{sign}(\theta_t-\theta_t') \mathcal{A}_{X_t,Y_t,N_0}$. In fact $z_t'+\textrm{sign}(\theta_t-\theta_t') \mathcal{A}_{X_t,Y_t,N_0}-z_t-\textrm{sign}(\theta_0-\theta_0') \mathcal{A}_{X_0,Y_0,N_0}$ is the signed swept area of an oriented loop starting from $Y_0$, following:
% 	\begin{itemize}
% 	    \item the path ${s\mapsto Y_s, s\leq t}$ from $Y_0$ to $Y_t$;
% 	    \item the geodesic joining $Y_t$ to $X_t$;
% 	    \item the path ${s\mapsto X_{t-s}, s\leq t}$ joining $X_t$ to $X_0$;
% 	    \item finally the geodesic joining $X_0$ to $Y_0$.
% 	\end{itemize} Thus, up to a constant, the quantity $z_t'-z_t+\textrm{sign}(\theta_t-\theta_t')  \mathcal{A}_{X_t,Y_t,N_0}$ is this signed swept area with whose sign changing when the paths of $Y_t$ and $X_t$ are crossing.
% 	\end{proof}

\section{Couplings on simply connected manifolds of constant curvature $k$ and variation of the swept area}\label{sec: Coupling on manifolds}
In all that follow we denote by $M$ the simply connected (oriented) Riemannian manifold of dimension $2$ and of constant curvature $k$ (endowed with the usual Riemannian metric) and $i(M)$ its injectivity radius.  
% Note that, in the case of the Heisenberg group we have the same remark but considering a Brownian motion on $\mathbb{R}^2$. Thus we are going to describe a coupling model on manifolds of dimension $2$ with constant curvature $k$, having, for $k=1$, coupling on $S^2$, for $k=-1$, coupling on $\mathbf{H}^2$ and even, for $k=0$, coupling on $\mathbb{R}^2$.
In Subsection \ref{subsec: Hessiennes}, we study the first and second order derivatives of the distance function and of the oriented swept area between two smooth curves in $M$.
%with respect to the basis of vector fields used in Proposition \ref{relations}. 
In Subsection \ref{subsec: CouplingModel} we explain and prove Proposition \ref{relations}. In Subsection \ref{subsec: co-adapted}, we give some examples of co-adapted couplings. Most of them will be of importance for the construction of the successful coupling (Section \ref{sec: Successful coupling}).
\subsection{First and second order derivatives}\label{subsec: Hessiennes}
Let $x$, $y\in M$ such that $0<\rho(x,y)<i(M)$. We can define:
\begin{itemize}
	\item $e_1^x(x,y)=\frac{exp_x^{-1}(y)}{\rho(x,y)}\in T_xM$
	\item 
	$e_1^y(x,y)$ the parallel transport of $e_1^x$ along the geodesic joining $x$ and $y$;
	\item $e_2^x(x,y)\in T_xM$ such that $\left(e_1^x(x,y),e_2^x(x,y)\right)$ is an orthonormal positive basis of $T_xM$;
	\item $e_2^y(x,y)\in T_yM$ such that $\left(e_1^y(x,y),e_2^y(x,y)\right)$ is an orthonormal positive basis of $T_yM$.
\end{itemize}
We first give a local expression of the oriented swept area between some paths on $M$ which coincides with the one used to define the process $({A}_t)_t$ in Definition \ref{def: aireBalayee1}.
We consider:
\begin{itemize}
    \item $\gamma_x$, $\gamma_y:[0,+\infty[\mapsto M$ two curves $\mathcal{C}^1$ on $M$ starting at $x$ and $y$ respectively;
    \item $(s,t)\in[0,1]\times[0,+\infty[\mapsto c(s,t)\in M$ such that for all $t\geq 0$, $s\mapsto c(s,t)$ is a geodesic starting at $\ga_x(t)$ and ending at $\ga_y(t)$.
\item Provided that $0<\rho\left(\gamma_x(t),\gamma_y(t)\right)<i(M)$  we can define $e_1(s,t):=\frac{\partial_s c(s,t)}{\rho\left(\gamma_x(t),\gamma_y(t)\right)}$ and $e_2(s,t)$ such that $(e_1(s,t),e_2(s,t))$ is an orthonormal positive basis of $T_{c(s,t)}M$ for each $s,t$. In particular $e_1(0,0)=e_1^x(x,y)$ and $e_1(1,0)=e_1^y(x,y)$.
\end{itemize}
Considering the positive local chart $(x^1,x^2)$ induced by the parametrization $(s,\zeta)\mapsto c(s,\zeta)$ on $M$, the volume form is given by $\sqrt{\det(G)}dx^1\wedge dx^2 $ where $G$ is the positive definite symmetric matrix representing the metric in these local coordinates. In particular, we have: \begin{equation*}\det(G)=||\partial_{s}c(s,\zeta)||^2||\partial_{\zeta}c(s,\zeta)||^2-\langle\partial_{s}c(s,\zeta),\partial_{\zeta}c(s,\zeta)\rangle^2=\det\big{(}\partial_{s}c(s,\zeta),\partial_{\zeta}c(s,\zeta)\big{)}^2
\end{equation*}
where the determinant is calculated in the basis $(e_1(s,\zeta),e_2(s,\zeta))$ of $T_{c(s,\zeta)}M$ for each $\zeta,t$.
Note that $x_1(c(s,\zeta))=s$ if and only if $\det\big{(}\partial_{s}c(s,\zeta),\partial_{\zeta}c(s,\zeta)\big{)}> 0$.
Then, the signed swept area can be locally defined by: \begin{equation}\label{FonctionAire}\mathcal{A}(t):=\int_0^1\int_{0}^t \det(\partial_s c(s,\zeta),\partial_{\zeta}c(s,\zeta))d\zeta ds.\end{equation} 
In particular there exists a vector field $V$ on $M\times M$ such that: \begin{equation}\label{eq: V}
\mathcal{A}(t)=\int_0^t\left\langle V_{\left(\gamma_1(\zeta),\gamma_2(\zeta)\right)},\left(\dot{\gamma}_1(t),\dot{\gamma}_2(t)\right)\right\rangle d\zeta.
\end{equation}

The following lemma gives the expressions of the first order derivative and the Hessian of the distance using the basis $\left(e_1^x(x,y),e_2^x(x,y)\right)$ and $\left(e_1^y(x,y),e_2^y(x,y)\right)$ defined above. Even though this result is well known, we will give its proof later in this subsection as it uses the same elements as to prove Lemma \ref{HessA}.
\begin{lemme}\label{cov rho}
	Let $x$, $y\in M$, $r=\rho(x,y)$ with $0<r<i(M)$, $u\in T_xM$ and $v\in T_y M$. Then, we have:
	\begin{equation*}d\rho_{(x,y)}(u,v)=v_1-u_1
	\end{equation*}
	and
	\begin{equation*}
		\Hess(\rho)_{(x,y)}((u,v),(u,v))=
			\sqrt{k}(u_2^2+v_2^2)\cot(\sqrt{k}r)-2\sqrt{k}u_2 v_2\frac{1}{\sin(\sqrt{k}r)}.
	\end{equation*}
	with $u_i=\langle u,e_i^x(x,y)\rangle$ and $v_i=\langle v, e_i^y(x,y)\rangle$ for $i\in\{1,2\}$.
\end{lemme}

We also give a result for what could be seen as the first order derivative and hessian of the signed swept area (if it was a function on $M\times M$). As far as the author knows, this has never been computed before.
\begin{lemme}\label{HessA}
	With the same hypothesis and notations than in Lemma \ref{cov rho} and $V$ defined in (\ref{eq: V}), we have:
	\begin{equation*}\left\langle V_{\left(x,y\right)},\left(u,v\right)\right\rangle=\frac{1}{\sqrt{k}}\tan\left(\frac{\sqrt{k}r}{2}\right)(u_2+v_2),
	\end{equation*}
	\begin{equation*}
		\left\langle \nabla _{\left(u,v\right)}V_{\left(x,y\right)},\left(u,v\right)\right\rangle=\frac{u_2v_1-v_2u_1}{\cos^2(\frac{\sqrt{k}r}{2})}+\tan^2\left(\frac{\sqrt{k}r}{2}\right)(v_2v_1-u_2u_1).
	\end{equation*}
\end{lemme}
Note that the above results are given as general formulas for the three cases $k>0$, $k<0$ and $k=0$ (taking the limits when $k$ tends to $0$ in the last case).

To prove these two lemmas, we need to evaluate in zero the first and second derivatives of the functions $r:t\mapsto \rho(\gamma_x(t),\gamma_y(t))$ and $\A$ with $\gamma_x(t):=\exp_x(tu)$ and $\gamma_y(t):=\exp_y(tv)$. We use the notations $c(s,t)$, $e_1(s,t)$, $e_2(s,t)$ to denote the same objects as previously. In particular, as $0<\rho(x,y)<i(M)$, there exists $t_0>0$ such that $(e_1(s,t),e_2(s,t))$ is well defined for $t\in[0,t_0]$. 

Let define $s\mapsto J(s,t):=\partial_t c(s,t)$. With this choice of $\gamma_x$ and $\gamma_y$, $s\mapsto J(s,t):=\partial_t c(s,t)$ is a Jacobi field for all $t\geq 0$. For $i\in\{1,2\}$, we denote:
\begin{equation*}
   u_i(t):=\langle J(0,t), e_i(0,t)\rangle ,\ u_i:=u_i(0)\text{ and } v_i(t):=\langle J(1,t), e_i(1,t)\rangle, \ v_i:=v_i(0).
\end{equation*} 
In the basis $(e_1(s,t),e_2(s,t))$ we have the following decomposition of this Jacobi field $J$:
\begin{lemme}\label{Lemme: jacobi}
	With the above notations and hypothesis we can write $J(s,t)=j_1(s,t)e_1(s,t)+j_2(s,t)e_2(s,t)$
	with:
	\begin{align}
		%j_1(s,t)&=u_1(t)(1-s)+v_1(t)s;\label{eq: J1}\\
		j_2(s,t)&=u_2(t)\cos(\sqrt{k}r(t)s)+\frac{v_2(t)-u_2(t)\cos(\sqrt{k}r(t))}{\sin(\sqrt{k}r(t))}\sin(\sqrt{k}r(t)s);\label{eq: J2}\\
		 u_2'(0)&=\sqrt{k}\frac{u_2\cos(\sqrt{k}r)-v_2}{\sin(\sqrt{k}r)}u_1 ;\  v_2'(0)=\sqrt{k}\frac{u_2-v_2\cos(\sqrt{k}r)}{\sin(\sqrt{k}r)}v_1.\label{eq: J3}
	\end{align}
	
\end{lemme}
Note that for all the following computations, we will use the notations $\nabla_s$ for $\nabla_{\partial_s c(s,t)}$ and $\nabla_t$ for $\nabla_{\partial_t c(s,t)}$. This lemma is the main result to prove Lemma \ref{cov rho} and Lemma \ref{HessA}.
\begin{proof}[Proof of Lemma \ref{Lemme: jacobi}]
	For $t$ constant, $s\mapsto e_i(s,t)$ is defined by parallel transport along $s\mapsto c(s,t)$. Thus $\nabla_s e_i(s,t)=0$ and we have: \begin{align*}
	\nabla_sJ(s,t)&=\partial_{s}j_1(s,t)e_1(s,t)+\partial_{s}j_2(s,t)e_2(s,t)\\  \text{ and }\nabla_s\nabla_sJ(s,t)&=\partial^2_{ss}j_1(s,t)e_1(s,t)+\partial^2_{ss}j_2(s,t)e_2(s,t).\end{align*}
	
	By definition of the Jacobi fields, we also have: \begin{align*}
	\nabla_s\nabla_sJ(s,t)&=-R(J(s,t),\partial_sc(s,t))\partial_s c(s,t)\\
%	&=-r(t)^2(j_1(s,t)R(e_1(s,t),e_1(s,t))e_1(s,t)+j_2(s,t)R(e_2(s,t),e_1(s,t))e_1(s,t))\\
	&=-kr(t)^2j_2(s,t)e_2(s,t).\end{align*}
	By solving the differential equation 
	%s $\Bigg{\{}\begin{array}{l}
%	\partial^2_{ss}j_1(s,t)=0\\
%	j_1(0,t)=u_1(t)\\ j_1(1,t)=v_1(t)
%	\end{array}$ and 
$\Bigg{\{}\begin{array}{l}
	\partial^2_{ss}j_2(s,t)=-kr(t)^2j_2(s,t)\\
	j_2(0,t)=u_2(t)\\ j_2(1,t)=v_2(t)
	\end{array}$ we obtain 
	%(\ref{eq: J1}) and 
	(\ref{eq: J2}).\\ 
We now just have to compute $u_2'(0)$ and $v_2'(0)$. As $u_2(t)=\langle \partial_t c(0,t),e_2(0,t)\rangle$ and $t\mapsto c(0,t)$ is a geodesic, we have:
\begin{equation*}
u_2'(t)=\langle\nabla_t\partial_t c(0,t),e_2(0,t)\rangle+\langle\partial_t c(0,t),\nabla_t e_2(0,t)\rangle=\langle\partial_t c(0,t),\nabla_t e_2(0,t)\rangle.\end{equation*}
Thus, $u_2'(0)=\langle u,\nabla_t e_2(0,t)_{|t=0}\rangle$ and, in the same way, we have $	v_2'(0)=\langle v,\nabla_t e_2(1,t)_{|t=0}\rangle$ .

Note that, as $\langle e_2(s,t),e_2(s,t)\rangle=1$ and $\langle e_2(s,t),e_1(s,t)\rangle=0$, we have:
\begin{align*}
\langle \nabla_t e_2(s,t),e_2(s,t)\rangle&=\frac{1}{2}\partial_t\langle e_2(s,t),e_2(s,t)\rangle=0\text{ and }\\
\langle \nabla_t e_2(s,t),e_1(s,t)\rangle&=\partial_t(\langle e_2(s,t),e_1(s,t)\rangle)-\langle e_2(s,t),\nabla_t e_1(s,t)\rangle=-\langle e_2(s,t),\nabla_t e_1(s,t)\rangle.
\end{align*}
Then, $\nabla_t e_2(s,t)=-\langle\nabla_t e_1(s,t),e_2(s,t)\rangle e_1(s,t)$. We have:
\begin{align*}
\nabla_te_1(s,t)&=\nabla_t\left(\partial_s c(s,t)\times\frac{1}{r(t)}\right)=-\frac{r'(t)}{r(t)^2}\partial_s c(s,t)+\nabla_t(\partial_s c(s,t))\times\frac{1}{r(t)}\\
&=-\frac{r'(t)}{r(t)} e_1(s,t)+\frac{\nabla_sJ(s,t)}{r(t)}\text{ since }\nabla\text{ is torsion free}.
\end{align*}
Thus: \begin{align}
\langle\nabla_t e_1(s,t),e_2(s,t)\rangle&=\frac{1}{r(t)}\partial_s j_2(s,t)\notag\\
&=\sqrt{k}\left(-u_2(t)\sin(\sqrt{k}r(t)s)+\frac{v_2(t)-u_2(t)\cos(\sqrt{k}r(t))}{\sin(\sqrt{k}r(t))}\cos(\sqrt{k}r(t)s)\right).\label{eq: grad e_1}\end{align}
We obtain: \begin{align*}&\nabla_te_2(0,t)=-\sqrt{k}\frac{v_2(t)-u_2(t)\cos(\sqrt{k}r(t))}{\sin(\sqrt{k}r(t))}e_1(0,t)\\
\text{ and } &\nabla_te_2(1,t)=-\sqrt{k}\frac{v_2(t)\cos(\sqrt{k}r(t))-u_2(t)}{\sin(\sqrt{k}r(t))}e_1(1,t).\end{align*} This gives (\ref{eq: J3}).

\end{proof}

%Although the computations has certainly been done elsewhere before, we want to describe the proof of Lemma \ref{cov rho} since some elements will be re-used to obtain the expression of the first order derivative and of the Hessian for the swept area in Lemma \ref{HessA}.
We can now give the proofs of the two Lemmas.
\begin{proof}[Proof of Lemma \ref{cov rho}]
	The general formulas for $L'(0)$ and $L''(0)$ are classical results that can be found in numerous books. As an example they are given in Lemma 4.1.1 and Theorem 4.1.1 in~\cite{Jost}.
	With our notations, we obtain:
	\begin{align}\label{Jost}
	r'(0)&=\left(\int_0^1\frac{\frac{d}{ds}\langle J(s,t),r(t)e_1(s,t)\rangle}{r(t)}-\frac{\langle J(s,t),r(t)\nabla_s e_1(s,t)\rangle}{r(t)}ds\right)_{|t=0}\\
	r''(0)&=\frac{1}{r}\left(\int_0^1\partial_s j_2(s,t)^2-r(t)^2j_2(s,t)^2\left\langle R(e_1(s,t),e_2(s,t))e_2(s,t),e_1(s,t)\right\rangle ds\right)_{|t=0}\notag\\
	&-\frac{1}{r}\left[\langle\nabla_t\partial_tc(s,t),r(t)e_1(s,t)\rangle_{t=0}\right]_0^1.
	\end{align}
	In particular, as said before, $\nabla_s e_1(s,0)=0$, thus $r'(0)=j_1(1,0)-j_1(0,0)=v_1-u_1$.\\
	We now look at $r''(0)$. We note that $t\mapsto c(s,t)$ is a geodesic for $s=0$ and $s=1$. Thus, we have $\nabla_t\partial_tc(s,t)=0$. As $\left(e_1(s,t),e_2(s,t)\right)$ is an orthonormal basis, \begin{equation*}\left\langle R(e_1(s,t),e_2(s,t))e_2(s,t),e_1(s,t)\right\rangle=k.\end{equation*}

	Thus, $r''(0)=\frac{1}{r}\int_0^1\partial_s j_2(s,0)^2-kr^2j_2(s,0)^2 ds$. Using, (\ref{eq: J2}) and (\ref{eq: grad e_1}) evaluated in $t=0$,
% 		\begin{equation*} \partial_sj_2(s,0)=\sqrt{k}r\left(-u_2\sin(\sqrt{k}rs)+\frac{v_2-u_2\cos(\sqrt{k}r)}{\sin(\sqrt{k}r)}\cos(\sqrt{k}rs)\right),\end{equation*}
		we get: \begin{align*}r''(0)&=kr\int_0^1\left(\sin^2(\sqrt{k}rs)-\cos^2(\sqrt{k}rs)\right)\left(u_2^2-\frac{(v_2-u_2\cos(\sqrt{k}r))^2}{\sin^2(\sqrt{k}r)}\right)\\
		&-4u_2\frac{v_2-u_2\cos(\sqrt{k}r)}{\sin(\sqrt{k}r)}\cos(\sqrt{k}rs)\sin(\sqrt{k}rs)ds.\end{align*}
		By integrating we finally obtain $r''(0)=\sqrt{k}\left(\left(u_2^2+v_2^2\right)\cot(\sqrt{k}r)-2\frac{u_2v_2}{\sin(\sqrt{k}r)}\right)$.
\end{proof}

\begin{proof}[Proof of Lemma \ref{HessA}]
We first explain why we just need to compute $\A'(0)$ and $\A''(0)$. As $\gamma_x$ and $\gamma_y$ are chosen to be two geodesics, we have $\nabla_t(\dot{\gamma}_x(t),\dot{\gamma}_y(t))=0$. Using relation (\ref{eq: V}), we then obtain:
\begin{equation*}\mathcal{A}'(t)=\langle V_{(\gamma_1(t),\gamma_2(t))},(\dot{\gamma}_1(t),\dot{\gamma}_2(t))\rangle\text{ and }
\mathcal{A}''(t)=\langle \nabla_{(\dot{\gamma}_1(t),\dot{\gamma}_2(t))}V_{(\gamma_1(t),\gamma_2(t))},(\dot{\gamma}_1(t),\dot{\gamma}_2(t))\rangle.
\end{equation*}
For $t=0$, we obtain $\left\langle V_{\left(x,y\right)},\left(u,v\right)\right\rangle$ and $\left\langle \nabla _{\left(u,v\right)}V_{\left(x,y\right)},\left(u,v\right)\right\rangle$.

We now use relation (\ref{FonctionAire}) to compute $\mathcal{A}'(0)$ and $\mathcal{A}''(0)$.
We have: 
\begin{align*}\A'(t)&=\int_0^1\det(\partial_{s}c(s,t),\partial_tc(s,t))ds=\int_0^1\det(r(t)e_1(s,t),J(s,t))ds=r(t)\int_0^1j_2(s,t)ds\\
		&=r(t)\left(\int_0^1 u_2(t)\cos(\sqrt{k}r(t)s)+\frac{v_2(t)-u_2(t)\cos(\sqrt{k}r(t))}{\sin(\sqrt{k}r(t))}\sin(\sqrt{k}r(t)s)ds\right)\\
%		&=\frac{1}{\sqrt{k}}\left( u_2(t)[\sin(\sqrt{k}r(t)s)]_0^1+\frac{v_2(t)-u_2(t)\cos(\sqrt{k}r(t))}{\sin(\sqrt{k}r(t))}[-\cos(\sqrt{k}r(t)s)]_0^1\right)\\
%		&=\frac{1}{\sqrt{k}}\left( u_2(t)\sin(\sqrt{k}r(t))+\frac{v_2(t)-u_2(t)\cos(\sqrt{k}r(t))}{\sin(\sqrt{k}r(t))}\left(1-\cos(\sqrt{k}r(t))\right)\right)\\
%		&=\frac{u_2(t)\sin^2(\sqrt{k}r(t))-u_2(t)\cos(\sqrt{k}r(t))(1-\cos(\sqrt{k}r(t)))+v_2(t)(1-\cos(\sqrt{k}r(t)))}{\sqrt{k}\sin(\sqrt{k}r(t))}\\ 
		&
%		=(u_2(t)+v_2(t))\frac{1-\cos(\sqrt{k}r(t))}{\sqrt{k}\sin(\sqrt{k}r(t))}
		=(u_2(t)+v_2(t))\frac{1}{\sqrt{k}}\tan\left(\frac{\sqrt{k}r(t)}{2}\right).
	\end{align*}
	We obtain the first expected result. This also leads to:
	\begin{equation*}
		\A''(t)=\left(u_2'(t)+v_2'(t)\right)\frac{1}{\sqrt{k}}\tan\left(\frac{\sqrt{k}r(t)}{2}\right)+\left(u_2(t)+v_2(t)\right)\frac{r'(t)}{2\cos^2(\frac{\sqrt{k}r(t)}{2})}.
	\end{equation*}
	Note that $r'(0)$ has been calculated for the proof of Lemma \ref{cov rho}, so:
	\begin{equation*}
		\A''(0)=(u_2'(0)+v_2'(0))\frac{1}{\sqrt{k}}\tan\left(\frac{\sqrt{k}r}{2}\right)+(u_2+v_2)\frac{v_1-u_1}{2\cos^2(\frac{\sqrt{k}r}{2})}.
	\end{equation*}
	Finally:
	\begin{align*}
		\A''(0)&=\left(\frac{u_2\cos(\sqrt{k}r)-v_2}{\sin(\sqrt{k}r)}u_1+\frac{u_2-v_2\cos(\sqrt{k}r)}{\sin(\sqrt{k}r)}v_1\right)\tan\left(\frac{\sqrt{k}r}{2}\right)+(u_2+v_2)\frac{v_1-u_1}{2\cos^2(\frac{\sqrt{k}r}{2})}\\
%		&=\left((u_2u_1-v_2v_1)\cot(\sqrt{k}r)+(u_2v_1-v_2u_1)\frac{1}{\sin(\sqrt{k}r)}\right)\tan\left(\frac{\sqrt{k}r}{2}\right)\\
%		& \ +\frac{u_2v_1-v_2u_1-(u_2u_1-v_2v_1)}{2\cos^2(\frac{\sqrt{k}r}{2})}\\
%		&=\left(u_2u_1-v_2v_1\right)\left(\cot(\sqrt{k}r)\tan\left(\frac{\sqrt{k}r}{2}\right)-\frac{1}{2\cos^2(\frac{\sqrt{k}r}{2})}\right)+\left(u_2v_1-v_2u_1\right)\left(\frac{\tan(\frac{\sqrt{k}r}{2})}{\sin(\sqrt{k}r)}\right.\\
%		& \ \left.+\frac{1}{2\cos^2(\frac{\sqrt{k}r}{2})}\right)\\
%		&=\left(u_2u_1-v_2v_1\right)\left(\frac{\cos(\sqrt{k}r)}{2\cos^2(\frac{\sqrt{k}r}{2})}-\frac{1}{2\cos^2(\frac{\sqrt{k}r}{2})}\right)+\left(u_2v_1-v_2u_1\right)\left(\frac{1}{2\cos^2(\frac{\sqrt{k}r}{2})}\right.\\
%		& \ \left.+\frac{1}{2\cos^2(\frac{\sqrt{k}r}{2})}\right).
%	\end{align*}
%	Thus 	$\A''(0)
&=	-\left(u_2u_1-v_2v_1\right)\tan^2\left(\frac{\sqrt{k}r}{2}\right)+\left(u_2v_1-v_2u_1\right)\frac{1}{\cos^2(\frac{\sqrt{k}r}{2})}.\end{align*}
\end{proof}
\subsection{Description of a coupling model}\label{subsec: CouplingModel}
	In this subsection we study the processes $R_t$ and $A_t$ for co-adapted couplings on the Riemannian manifold $M$. According to Proposition \ref{distanceB}, this will be enough to induce and characterize a large-scale of co-adapted coupling on the subRiemannian manifold $SU(2)$ and $SL(2,\mathbb{R})$. We first recall the general definition of such a coupling.
\begin{defi}
	Let $\mu_t$ and $\nu_t$ the laws at time $t$ of two continuous Markov processes  $(X_t)_t$ and $(Y_t)_t$ on $M$. The process $(\tilde{X}_t,\tilde{Y}_t)_t$ is said a co-adapted coupling of $(X_t)_t$ and $(Y_t)_t$ if:
	\begin{itemize}
		\item the processes $(\tilde{X}_t)_t$ and $(\tilde{Y}_t)_t$ lay in a same filtered probability space $(\mathcal{F}_t)_t$ and follow in each time $t$ the probability laws $\mu_t$ and $\nu_t$ respectively;
		\item for all bounded measurable function $f$, $z\in M$, $s$, $t>0$, the functions\\ $z\mapsto\esp[f(\tilde{X}_{t+s}) \ | \ \mathcal{F}_s, \ \tilde{X}_s=z]$ and $z\mapsto\esp[f(X_{t+s}) \ | \ {X}_s=z]$ (resp. \\$z\mapsto\esp[f(\tilde{Y}_{t+s}) \ | \ \mathcal{F}_s, \ \tilde{Y}_s=z]$ and $z\mapsto\esp[f(Y_{t+s}) \ | \ {Y}_s=z]$) are equal $\pr_{X_s}$-a.s. (resp. $\pr_{Y_s}$-a.s.).
	\end{itemize}
\end{defi}
We can now prove Proposition \ref{relations}.
	\begin{proof}
		As we define the processes until the first time $T$ where $R_t\in\{0,i(M)\}$, $e_1^X(t)$ and $e_1^Y(t)$ are well defined on $[0,T[$. By construction $U$, $V$, $e_1^X$, $e_2^X$, $e_1^Y$ and $e_2^Y$, $X$ and $Y$ are adapted to the same filtration. 
		%Then, using Proposition \ref{Emery}, 
		Thus the processes $(X_t)_{0\leq t\leq T}$ and $(Y_t)_{0\leq t\leq T}$ are Brownian motions on $M$ and we obtain a co-adapted coupling.
		
		To deal with the second part of the proof, we just need to notice that $e_i^X(t)=e_i^{X_t}(X_t,Y_t)$ and $e_i^Y(t)=e_i^{Y_t}(X_t,Y_t)$ for all $i\in\{1,2\}$. We obtain the equation (\ref{eq: R(t)}) by using Lemma \ref{cov rho} and Itô's formula:
		\begin{equation*}
			dR_t=d\rho_{(X_t,Y_t)}(d^{\nabla}X_t,d^{\nabla}Y_t)+\frac{1}{2}\Hess(\rho)(X_t,Y_t)((d^{\nabla}X_t,d^{\nabla}Y_t),(d^{\nabla}X_t,d^{\nabla}Y_t))
			\end{equation*}
		We now deal with Equation (\ref{eq: A(t)}). As (\ref{eq: V}) is true for every $\mathcal{C}^1$ paths $\gamma_1$ and $\gamma_2$, we have:
		\begin{align*}
		    A_t&=\int_0^t \left\langle V_{(X_{\zeta},Y_{\zeta})},\circ d(X_{\zeta},Y_{\zeta})\right\rangle \text{ with }\circ d\text{ the Stratonovich derivative}\\
		 &=\int_0^t \left\langle V_{(X_{\zeta},Y_{\zeta})}, d^{\nabla}(X_{\zeta},Y_{\zeta})\right\rangle+\frac{1}{2}\int_0^t\left\langle \nabla_{d^{\nabla}(X_{\zeta},Y_{\zeta})}V_{(X_{\zeta},Y_{\zeta})},d^{\nabla}(X_{\zeta},Y_{\zeta})\right\rangle .\end{align*}
		
		Using Lemma \ref{HessA}, we obtain the expected result. The last relation in Proposition \ref{relations} is immediate.
		\end{proof}
% 			 We can now define couplings $(X_t,Y_t)$ of Brownian motions on $M$ using two dimensional Brownian motions $(U_1(t),U_2(t))$ and $(V_1(t),V_2(t))$ adapted to a common filtration such that we have relation (\ref{equa}):
% 	\begin{equation*}
% 	d^{\nabla}X_t=dU_1(t)e_1^X(t)+dU_2(t)e_2^X(t) \text{ and } d^{\nabla}Y_t=dV_1(t)e_1^Y(t)+dV_2(t)e_2^Y(t)	
% 	\end{equation*} with $e_i^X(t)$ and $e_i^Y(t)$ defined as follow. We consider the continuous stochastic process $R_t=\rho(X_t,Y_t)$ with $\rho$ the Riemannian distance. We suppose that $0<R_t<i(M)$ for all $t\in I$, with $i(M)$ the injectivity radius of $M$ and $I$ an open interval on $\mathbb{R}^{+}$. Then for all $t\in I$ we can define:
%	\begin{NB}
%	The restriction on the starting point of the Brownian motions assure the existence and uniqueness of vector $e_1$. For example, if $M$ is the three dimensional sphere, we have $i(M)=\pi$ and so we must have $0<\rho(x,y)<\pi$ to uniquely define these couplings.\\
%	In the case when $x=y$ we can choose for $e_1$ any unitary vector in $T_xM$ for $x$; in the case when $\rho(x,y)=i(M)$ we have to choose one element in $\exp^{-1}({b})$.
%\end{NB}
\begin{NB}\label{distanceNulle}
In fact, by describing the covariant derivative of $(e_i^X(t))_t$ (resp. $(e_i^Y(t))_t$) along $(X_t)_t$ (resp. $(Y_t)_t$) we can define our orthonormal basis as the solutions of a system of stochastic differential equations depending on $U_1, U_2, V_1, V_2$ and $R_t$. In the case where no singularities appear in this system, our coupling keeps sense even if $R_t\in\{0,i(M)\}$. Then we can consider $R_t$ as a signed distance and the results of Proposition \ref{relations} are still true.
\end{NB}

In $\mathbb{R}^2$, there exists a general method to construct co-adapted couplings. The proof of the following proposition can be found in~\cite{kendall2009brownian} for example.
	\begin{prop}\label{co-adapt}
		We consider the filtered probability space $(\Omega, (\mathcal{F}_t)_t,\mathbb{P})$. The following assertions are equivalent:
		\begin{enumerate}[label=(\roman*)]
		    \item $(U,V)$ is a co-adapted coupling of Brownian motions in $\mathbb{R}^2$;
		\item Enlarging the filtration if needed, there exists a two dimensional Brownian motion $W$ adapted to the filtration and independent of $U$, and $K(t)$, $\hat{K}(t)\in \mathcal{M}_2(\mathbb{R})$ with $K(t)K(t)^T+\hat{K}(t)\hat{K}(t)^T=I_2$, $K(t)$, $\hat{K}(t)\in \mathcal{F}_t$  such that :\begin{equation}\label{co-adapt2}
			dV(t)=K(t) dU(t)+\hat{K}(t) dW(t).
		\end{equation}
		\end{enumerate}
	\end{prop}
	
%	\begin{proof}
%	   The equivalence of (ii) and (iii) is a well known fact about couplings in $\mathbb{R}^2$ that can be found in~\cite{kendall2009brownian} for example.\\
%	    Let us show that (ii) implies (i). The converse implication may be dealt with the same method, also used in~\cite{pascu2018couplings}. We denote $\mathcal{G}$ the filtration induced by $U$ and $V$. As said in a previous remark, our basis are solutions of equations depending of $R$, $U$ and $V$. Then, $X_t, Y_t, (e_1^X(t),e_2^X(t))$ and $(e_1^Y(t),e_2^Y(t))$ are adapted to the filtration $\mathcal{G}$. Moreover, $X_t$ and $Y_t$ are Markov processes for this same filtration. Then, this is also true for the filtration $\mathcal{F}$ induced by $X_t$ and $Y_t$ as it is included in $\mathcal{G}$. Thus, for $s,t>0$, $z\in M$, we have \begin{align*}
%	        \esp[f(X_{t+s})|\mathcal{F}_s, X_s=z]&=\esp[\esp [f(X_{t+s})|\mathcal{G}_s, X_s=z]|\mathcal{F}_s, X_s=z]=\esp[P_{s,t}f(z)|\mathcal{F}_s, X_s=z]\\
%	        &=P_{s,t}f(z)
%	    \end{align*} with $P_{s,t}$ the transition function of $X$. Thus we obtain a co-adapted coupling for $(X,Y)$.
%	\end{proof}
%	
	Using (\ref{co-adapt2}), we can rewrite our relations for $R_t$ and $A_t$:
	\begin{prop}
		For a co-adapted coupling satisfying Proposition \ref{relations} and relation (\ref{co-adapt2}), we have:
		\begin{align*}
			dR_t\cdot dR_t&=2(1-K_{1,1}(t))dt;\\
			dR_t&\stackrel{(m)}{=}\sqrt{k}\cot(\sqrt{k}R_t)dt-\frac{\sqrt{k}}{\sin(\sqrt{k}R_t)}K_{2,2}(t)dt=\sqrt{k}\frac{\cos(\sqrt{k}R_t)-K_{2,2}(t)}{\sin(\sqrt{k}R_t)}dt;\\
			dA_t\cdot dA_t&=2\frac{\tan^2(\frac{\sqrt{k}R_t}{2})}{k}(1+K_{2,2}(t))dt;\\
			dA_t&\stackrel{(m)}{=}\frac{1}{2\cos^2(\frac{\sqrt{k}R_t}{2})}(K_{1,2}(t)- K_{2,1}(t));\\
			dR_t\cdot dA_t&=\frac{1}{\sqrt{k}}\tan\left(\frac{\sqrt{k}R_t}{2}\right)( K_{1,2}(t)-K_{2,1}(t))dt.
		\end{align*}
	\end{prop}
	\subsection{Some examples of co-adapted couplings on the Riemannian manifold $M$}\label{subsec: co-adapted}
%	\begin{ex}
In this subsection, we give some examples of co-adapted couplings on $M$ constructed by using (\ref{co-adapt2}). As explained in the Proposition \ref{distanceB}, all these constructions induce co-adapted couplings of Brownian motions on the subRiemannian manifolds $\hh$, $SU(2)$ and $SL(2,\mathbb{R})$.
		\begin{itemize}
			\item \underline{Synchronous coupling:} We take $K(t)=I_2$, $\hat{K}(t)=0$, $dV(t)=dU(t)$. We get:
			\begin{align*}
			    &dR_t\cdot dR_t=0;\\
			    &Drift(dR_t)=\sqrt{k}\frac{\cos(\sqrt{k}R_t)-1}{\sin(\sqrt{k}R_t)}dt=-\sqrt{k}\times \tan\left(\frac{\sqrt{k}R_t}{2}\right)dt.
			\end{align*} Thus $R_t$ is deterministic and $R_t=\frac{2}{\sqrt{k}}\arcsin\left(e^{-\frac{kt}{2}}\sin(\frac{\sqrt{k}R_0}{2})\right)$.	We also have:
		 \begin{equation*} dA_t\cdot dA_t=4\frac{\tan^2(\frac{\sqrt{k}R_t}{2})}{k}dt\text{ and }
			Drift(dA_t)=0.\end{equation*}
			In particular, $A_t$ is a martingale.\\
			
			\item \underline{Reflection coupling:}	For $K(t)=\begin{pmatrix}
				-1 & 0 \\
				0 & 1
			\end{pmatrix}$, $\hat{K}(t)=0$, $dV_1(t)=-dU_1(t)$ and $dV_2(t)=dU_2(t)$, we get:
			\begin{equation*}
			    dR_t\cdot dR_t=4dt\text{ and }Drift(dR_t)=-\sqrt{k}\times \tan\left(\frac{\sqrt{k}R_t}{2}\right).
			\end{equation*}
			As before, we have:  \begin{equation*} dA_t\cdot dA_t=4\frac{\tan^2\left(\frac{\sqrt{k}R_t}{2}\right)}{k}dt\text{ and }
			Drift(dA_t)=0.\end{equation*}
			Note that, using Remark \ref{distanceNulle}, we can show that $R_t$ keeps sense even if $R_t=0$.
			\item \underline{Perverse coupling:} (We use here the terminology introduced by Kendall in~\cite{kendall2009brownian} to name some shy couplings of Brownian motions on $\mathbb{R}^2$.)\\
			For $K(t)=\begin{pmatrix}
				1 & 0 \\
				0 & -1
			\end{pmatrix}$, $\hat{K}(t)=0$, $dV_1(t)=dU_1(t)$ and $dV_2(t)=-dU_2(t)$, we get:
			\begin{align*}
			&dR_t\cdot dR_t=0;\\ &Drift(dR_t)=\sqrt{k}\frac{\cos(\sqrt{k}R_t)+1}{\sin(\sqrt{k}R_t)}dt=\sqrt{k}\times \cot\left(\frac{\sqrt{k}R_t}{2}\right).
			\end{align*}
			Thus $R_t$ is deterministic and $R_t=\frac{2}{\sqrt{k}}\arccos\left(e^{-\frac{kt}{2}}\cos\left(\frac{\sqrt{k}R_0}{2}\right)\right)$ for $k\neq0$. We also have $A_t$ constant.
		\end{itemize}
		In all these examples, $dR_t\cdot dA_t=0$.\\
%	\end{ex}
	We also can add a noise to these couplings in order to remove the drift part. In particular, this will permit to obtain a coupling with constant distance between the Brownian motions.
	\begin{itemize}
		\item \underline{Synchronous coupling with noise/ fixed-distance coupling:}\\\\
		Taking $K(t)=\begin{pmatrix}
			1 & 0 \\
			0 & \cos(\sqrt{k}R_t)\end{pmatrix}$, $\hat{K}(t)=\begin{pmatrix}
			0 & 0 \\
			0 & \sin(\sqrt{k}R_t)\end{pmatrix}$, we get $R_t$ constant. We also have: 
		\begin{align*}
		dA_t\cdot dA_t&=2\frac{\tan^2(\frac{\sqrt{k}R_t}{2})}{k}(1+\cos(\sqrt{k}R_t))dt=\frac{4}{k}\sin^2\left(\frac{\sqrt{k}R_t}{2}\right)dt\\
		&=\frac{4}{k}\sin^2\left(\frac{\sqrt{k}R_0}{2}\right)dt;\\
		Drift(dA_t)&=0.\end{align*}
		Note that $A_t$ is a Brownian motion up to a multiplicative constant.\end{itemize}
		We can do the same for the reflection coupling.
		\begin{itemize}
		\item \underline{Reflection coupling with noise:}	\\
		For $K(t)=\begin{pmatrix}
			-1 & 0 \\
			0 & \cos(\sqrt{k}R_t)
		\end{pmatrix}$, $\hat{K}(t)=\begin{pmatrix}
			0 & 0 \\
			0 & \sin(\sqrt{k}R_t)\end{pmatrix}$, we get:
			\begin{equation*} dR_t\cdot dR_t=4dt\text{ and }Drift(dR_t)=0.\end{equation*} 
			Thus $\frac{1}{2}R_t$ is a Brownian motion. We also have: \begin{equation*}dA_t\cdot dA_t=\frac{4}{k}\sin^2\left(\frac{\sqrt{k}R_t}{2}\right)dt\text{ and }
		Drift(dA_t)=0.\end{equation*}
	\end{itemize}
	For these two couplings, $dR_t\cdot dA_t=0$ too.
	\begin{NB}
		\begin{itemize}
			\item Let $x,y\in M$ (with $0<\rho(x,y)<i(M)$ if $k>0$) and $f:[0,\infty[\to[0,\infty[$ a positive (deterministic) function such that $f(0)=\rho(x,y)$. In~\cite{pascu2018couplings}, Pascu and Popescu proved that there exists a co-adapted coupling of Brownian motions $(X_t,Y_t)$ starting at $(x,y)$ such that $R_t=f(t)$ if and only if $f$ is continuous a.s. and satisfies, for a.e. $t\geq 0$, the differential inequality: \begin{equation*}-\sqrt{k}\tan\left(\sqrt{k}\frac{f(t)}{2}\right)\leq f'(t)\leq \sqrt{k}\cot\left(\frac{\sqrt{k}f(t)}{2}\right).\end{equation*}
			In particular the synchronous coupling and perverse coupling described above are the couplings realizing the extrema of this inequality.
	\item In the case where $k=0$ we find all the expected results for the Heisenberg group. See for example~\cite{kendall2007coupling,kendall2009brownian,bonnefont2018couplings}.
	\end{itemize}
	\end{NB}
	\section{Successful couplings on $SU(2)$}\label{sec: Successful coupling}
	We are now interesting ourselves in the construction of the successful coupling on $SU(2)$. To begin with, we construct a successful coupling of Brownian motions on $M$, a Riemannian manifold with curvature $k>0$, and their swept areas for the areas lying in $\mathbb{R}$. As in Section \ref{sec: Coupling on manifolds}, $i(M)$ denotes the injectivity radius of the manifold $M$. In particular, since $k>0$, we have $i(M)=\frac{\pi}{\sqrt{k}}$.
	
	\begin{thm}\label{thm: couplingAreaInR}
	Let $M$ be a Riemannian manifold of constant curvature $k>0$, $x,y\in M$ and $a\in\mathbb{R}$. There exists a coupling $(X_t,Y_t)$ of Brownian motions on $M$ starting from $(x,y)$ such that $\tau:=\inf\{t>0| R_t=0 \text{ and } A_t=a\}$ is a.s. finite with $A_t$ the signed swept area (started from $0$) defined in Definition \ref{def: aireBalayee1}.
	\end{thm}

	As $\zeta\equiv A_t+z_0'-z_0+\sgn(\theta_0-\theta_0')\A_{Y_0,X_0,N}\mod (4\pi)$, the coupling of Theorem \ref{thm: couplingAreaInR} for $a=z_0'-z_0+\sgn(\theta_0-\theta_0')\A_{Y_0,X_0,N_0}$ also induces a successful coupling on $SU(2)$. This proves Theorem \ref{successful}. In fact, using the compactness of $SU(2)$, the last coordinate $z_t$ of any Brownian motion lies in $\mathbb{R}/]-2\pi,2\pi]$. Then we can change some steps of the strategy of Theorem \ref{thm: couplingAreaInR} to obtain a faster coupling. This will be commented at the end of this section in Remark \ref{NB: compactness}.

We first explain the global strategy of the coupling from Theorem \ref{thm: couplingAreaInR}.
%  	Using the previous notations, let us suppose that $k>0$. 
As said before, our method is based on the idea from Kendall for coupling two dimensional real Brownian motions and their swept areas (\cite{kendall2007coupling}). The original idea is to switch between reflection and synchronous coupling, using reflection coupling to make $R_t$ decrease, and synchronous coupling to keep the swept area comparable to $R_t^2$ and decreasing as well.\\
	Here, in comparison to Kendall's original proposition, we will add a noise during the "synchronous coupling step" as in the subsection \ref{subsec: co-adapted} in order to keep $R_t$ constant. If not, we would have a strictly positive probability to be trapped in a "synchronous coupling step" without returning in a "reflection coupling step". Note that, for $k\rightarrow 0$, i.e., for real two dimensional Brownian motion, fixed-distance coupling (synchronous coupling with noise) and synchronous coupling describe in fact the same coupling.

	%We first suppose that $R_t<i(M)$ for all Firstly, we can suppose that $R_0>0$ using a reflection coupling if needed as $R_t$ is a time-changed Brownian motion during reflection coupling and doing it take a.s. finite time. The same way we can use synchronous coupling to obtain $A_t=0$ in an a.s. finite time without changing the value of $R_t$.
	
% 	Let us choose $\kappa$, $\eps>0$ such that $0<\eps<\kappa$. We denote $\tau$ the first time of coupling of the two Brownian motions AND their swept areas:
% 	\begin{equation*}
% 	    \tau:=\inf\{t>0| R_t=0 \text{ and } A_t=0\}.
% 	\end{equation*} 
We first make and explain some hypothesis on the initial parameters.
\begin{itemize}
    \item We can first suppose that $0<R_0<i(M)$. Indeed, by taking $t_0>0$ any deterministic positive real, we can define the coupling $(X_t,Y_t)_{t\in[0,t_0]}$ such that $X_t$ and $Y_t$ are independent. Then we have $0<R_{t_0}<i(M)$ a.s.. 
\item Without loss of generality, we can suppose that $a=0$ by considering the process $A_t-a$ instead of $A_t$. 
\item Using constant fixed-distance coupling if necessary, we can suppose that $A_0=0$ without changing the value of $R_0$ (indeed, $A_t$ is a changed-time Brownian motion during the fixed-distance coupling).
\end{itemize}

Let $\eta>0$ such that $i(M)-2\eta>R_0>0$. We define another stopping time:
	\begin{equation*}\tau_{\eta}:=\inf\{t>0 | R_t\geq i(M)-\eta\}.\end{equation*}
	We are going to study first $\tau\wedge\tau_{\eta}$ instead of $\tau$: this way we will have $R_t<i(M)-\eta$ for all $t<\tau$. We define $\kappa$ and $\eps$ such that $\kappa>\eps>0$.
	
 We will construct the coupling as follow on $[0,\tau\wedge\tau_{\eta}]$ :
	\begin{enumerate}
		\item\label{rfx0} We use the reflection coupling until the process $\frac{|A_t|}{R_t^2}$ starting at $0$ takes the value $\kappa$;
		\item\label{sy} While the process $\frac{|A_t|}{R_t^2}$, starting at $\kappa$, satisfies $\frac{|A_t|}{R_t^2}>\kappa-\epsilon$ we use the fixed-distance coupling;
		\item\label{rfx} While the process $\frac{|A_t|}{R_t^2}$, starting at $\kappa-\eps$, satisfies $\frac{|A_t|}{R_t^2}<\kappa$ we use the reflection coupling.
	\end{enumerate}
	We iterate the steps \ref{sy}. and \ref{rfx}. until $R_t=A_t=0$ or $R_t=i(M)-\eta$.
	%Note that is $R_t=i(M)$ the different coupling aren't uniquely defined.
	\begin{prop}\label{Kendall}
		Under the hypothesis $i(M)-2\eta>R_0>0$, the co-adapted coupling described below satisfies $\tau\wedge\tau_{\eta}<+\infty$ a.s. for $k> 0$. Moreover, we get $\pr(\tau>\tau_{\eta})<1$.
	\end{prop}
	\newcommand{\Ne}{N^{(\epsilon)}}
	\begin{proof}[Proof of Proposition \ref{Kendall}]
		Let us denote $\tau':=\tau\wedge\tau_{\eta}$. We define $\Ne:[0,\tau']\to\{0,1\}$ such that $\Ne(t)=0$ during fixed-distance coupling and $\Ne(t)=1$ during reflection coupling.
		We get, for all $t>0$: 
		\begin{align*}dR_t&=2\Ne(t)dC_t-\sqrt{k}\tan\left(\frac{\sqrt{k}R_t}{2}\right)\Ne(t)dt\\ 
		dA_t&=\frac{2}{\sqrt{k}}\times\frac{\sin(\frac{\sqrt{k}R_t}{2})}{\cos(\Ne(t)\frac{\sqrt{k}R_t}{2})}d\tilde{C}_t
		\end{align*} with $C_t$ and $\tilde{C}_t$ two independent Brownian motions in $\mathbb{R}$.
		Note that, during reflection coupling (i.e., when $\Ne(t)=1$) we have $\frac{|A_t|}{R_t^2}\leq \kappa$, whereas during fixed-distance coupling (i.e., $\Ne(t)=0$) we have $\frac{|A_t|}{R_t^2}\geq \kappa-\eps$. As $R$ varies only during reflection couplings, if $R_t\rightarrow 0$, we have $R_t^2\geq \frac{1}{\kappa}|A_t|$, thus $A_t\rightarrow 0$. In fact, with this strategy we have $\tau'=\inf\{t \ | \ R_t\in\{0,i(M)-\eta\} \}$.
		%Then we have to prove that $R_t$ hits $0$ or $i(M)-\eta$ in an a.s. finite time. 
		For the following computation, we take $t<\tau'$. This way we have $0<R_t<i(M)-\eta$.\\\\
		Let us define $\sigma(t):=\int_0^t \frac{4}{R_s^2}ds$ and $K_{\sigma(t)}=\ln(R_t)$. Using Itô's formula, we get: %$dK_{\sigma(t)}=\frac{dR_t}{R_t}-\frac{1}{2R_t^2}d\langle R_t, R_t\rangle$. Thus: 
		\begin{equation*}
		    d K_{\sigma(t)}\cdot dK_{\sigma(t)}=\frac{4\Ne(t)}{R_t^2}dt=\Ne(t) d\sigma(t)
		\end{equation*}
		\begin{align*}
		Drift(dK_{\sigma(t)})&=-\frac{\sqrt{k}}{R_t}\tan\left(\frac{\sqrt{k}R_t}{2}\right)\Ne(t) dt-\frac{2\Ne(t)}{R_t^2}dt\\
		&=-\Ne(t)\left(\frac{\sqrt{k}R_t \tan\left(\frac{\sqrt{k}R_t}{2}\right)}{4}+\frac{1}{2}\right)d\sigma(t).
		\end{align*} 
		With this change of time, excluding the times when $\Ne(t)=0$ (fixed-distance coupling) during which $K$ stays constant, $\sigma\mapsto K_{\sigma}$ acts like a Brownian motion with a negative drift.\\
		Let us also define $W_{\sigma(t)}=\frac{A_t}{R^2_t}$. Using Itô's formula, we get:
		%\begin{equation*} dW_{\sigma(t)}=\frac{1}{R_t^2}dA_t-2\frac{A_t}{R_t^3}dR_t+\frac{1}{2}\times 6\frac{A_t}{R_t^4}d\langle R_t, R_t\rangle=\frac{1}{R_t^2}dA_t-2\frac{A_t}{R_t^3}dR_t+3\times 4\Ne\frac{A_t}{R_t^4}dt.\end{equation*}
	%	Thus we have:
		\begin{align*}
		d W_{\sigma(t)}\cdot dW_{\sigma(t)}&=\left(\frac{4\sin^2\left(\frac{\sqrt{k}R_t}{2}\right)}{kR_t^4\cos^2\left(\Ne(t)\frac{\sqrt{k}R_t}{2}\right)}+\frac{16\Ne A_t^2}{R_t^6}\right)dt\text{ as }\langle dC_t, d\tilde{C}_t\rangle=0\\
	   % &=\left(\frac{\sin^2\left(\frac{\sqrt{k}R_t}{2}\right)}{kR_t^2\cos^2\left(\Ne(t)\frac{\sqrt{k}R_t}{2}\right)}+\frac{4\Ne A_t^2}{R_t^4}\right)d\sigma(t)\\
	    &=\left(\frac{1}{4}\left(\frac{2}{\sqrt{k}R_t}\right)^2\frac{\sin^2\left(\frac{\sqrt{k}R_t}{2}\right)}{\cos^2\left(\Ne\frac{\sqrt{k}R_t}{2}\right)}+4\Ne(t) W_{\sigma(t)}^2\right)d\sigma(t);\\
	    Drift(dW_{\sigma(t)})&=\left(\frac{2\sqrt{k}}{R_t^3} \tan\left(\frac{\sqrt{k}R_t}{2}\right)\Ne(t)A_t+3\times 4\Ne(t)\frac{A_t}{R_t^4}\right)dt\\
		&=\Ne(t) W_{\sigma(t)}\left(\frac{\sqrt{k}R_t}{2} \tan\left(\frac{\sqrt{k}R_t}{2}\right)+3\right) \ d\sigma(t).
	    \end{align*}
		Finally, $dK_{\sigma(t)}\cdot dW_{\sigma(t)}
		%=\frac{1}{R_t}dR_t\cdot\left(-2\frac{A_t}{R_t^3}dR_t\right)
		=-2\frac{A_t}{R_t^4}\times 4\Ne(t) dt=-2\Ne(t) W_{\sigma(t)} \ d\sigma(t)$.\\
		
		Let $S:=\sigma(\tau')$. For the following part, in order to simplify the notations, we will denote $\sigma$ instead of $\sigma(t)$ and $\Ne(\sigma)$ instead of $\Ne(t)$.
		We want to show that $\tau'<+\infty$ a.s.. As $\int_0^{S} e^{2K_{\sigma}}d\sigma=\int_0^\tau 4e^{2\log(R_t)}\frac{dt}{R_t^2}=4\tau'$, this is the same as showing that $\int_0^{S} e^{2K_{\sigma}}d\sigma<+\infty$ a.s..\\
		First of all let us note that the number of changes of type of coupling (reflection coupling/ fixed-distance coupling) is countable as it is finite on all closed and bounded interval of time  $[S_1,S_2]$ such that $S_2<\tau'$. Actually if we have an infinite number of changes of type of coupling, we can define (considering the time scale induced by $\sigma$) two sequences $(\sigma_n^s)_n$ and $(\sigma_n^r)_n$ such that $\sigma_{n}^s<\sigma_n^r<\sigma_{n+1}^s$ for all $n\geq 0$ and such that $\Ne(\sigma)=0$ on $[\sigma_n^s,\sigma_n^r[$ and $\Ne(\sigma_n^r)=1$. As $S_2<\tau'$, we have $R_t\neq 0$ and $W_{\sigma}$ well defined on $[S_1,S_2]$ such that $|W_{\sigma_n^s}|=\kappa$ and $|W_{\sigma_n^r}|=\kappa-\eps$. As $(\sigma_n^s)_n$ and $(\sigma_n^r)_n$ converge to the same limit and as $|W|$ is continuous, this leads to a contradiction.
		
		Thus we have a countable number of changes of type of coupling. Using the previous notations, let us denote by $[\sigma_n^s,\sigma_n^r[$ the intervals during which $\Ne(\sigma)=0$ and by $[\sigma_n^r,\sigma_{n+1}^s[$ the intervals during which $\Ne(\sigma)=1$. The same way we denote $t_n^s$ and $t_n^r$ such that $\sigma(t_n^s)=\sigma_n^s$ and $\sigma(t_n^r)=\sigma_n^r$.\\
		As seen before, $(R_t)_t$ is constant on intervals $[t_n^s,t_n^r[$ and acting as a two times Brownian motion with negative drift out of these intervals. As the time needed to exit a bounded open interval for a real Brownian motion with negative drift is a.s. finite, we have $\sum\limits_{n\geq 0}(t_{n+1}^s-t_{n}^r)<+\infty$~{a.s..} Note that: 
		\begin{align*}\int_0^S \Ne(\sigma) e^{2K_{\sigma}}d\sigma&
		%=\sum\limits_{n\geq 0}\int_{\sigma_n^r}^{\sigma_{n+1}^s}e^{2K_{\sigma}}d\sigma
		=\sum\limits_{n\geq 0}\int_{\sigma_n^r}^{\sigma_{n+1}^s}R^2_t d\sigma=\sum\limits_{n\geq 0}\int_{t_n^r}^{t_{n+1}^s}4dt=4\sum\limits_{n\geq 0}(t_{n+1}^s-t_{n}^r).\end{align*}
		Then, this quantity is a.s. finite and, in order to show that $\tau'<+\infty$, it is enough to show that: \begin{equation*}\int_0^S(1-\Ne)e^{2K_{\sigma}}d\sigma<+\infty.\end{equation*}
		As $(R_t)_t$ and $(K_\sigma)_{\sigma}$ are constant during fixed-distance coupling, we have: \begin{equation*}\int_0^S(1-\Ne)e^{2K_{\sigma}}d\sigma=\sum\limits_{n\geq 0}\int_{\sigma_n^s}^{\sigma_n^r}e^{2K_\sigma}d\sigma=\sum\limits_{n\geq 0}e^{2K_{\sigma_n^s}}(\sigma_n^r-\sigma_n^s).\end{equation*} Thus it is enough to show that $\sum\limits_{n\geq 0}e^{2K_{\sigma_n^s}}(\sigma_n^r-\sigma_n^s)<+\infty$ a.s..
		
		\begin{itemize}
			\item We first show the equality: 
			\begin{equation}\label{esp1}
				\esp\left[\exp\bigg(-\sum\limits_{n\geq 0}e^{2K_{\sigma_n^s}}(\sigma_n^r-\sigma_n^s)\bigg)\big\rvert(K_{\sigma_m^s})_m\right]=\exp\left(-\sum\limits_{n\geq 0}e^{K_{\sigma_n^s}}\times\frac{\frac{\sqrt{k}R_{t_n^s}}{2}}{\sin\left(\frac{\sqrt{k}R_{t_n^s}}{2}\right)}\times 2\sqrt{2}\eps\right).
			\end{equation}
			As our coupling is the fixed-distance one on $[\sigma_n^s,\sigma_n^r]$, there exists $V_\sigma$, a real Brownian motion, such that:
			\begin{equation*}dW_\sigma=\frac{1}{\sqrt{k}R_t}\sin\left(\frac{\sqrt{k}R_t}{2}\right)dV_\sigma=\frac{1}{\sqrt{k}R_{t_n^s}}\sin\left(\frac{\sqrt{k}R_{t_n^s}}{2}\right)dV_\sigma.\end{equation*} 
			In particular $|W_\sigma|$ only depends on $R_{t_n^s}$ for $\sigma_n^r\geq \sigma\geq \sigma_n^s$ and so depends on $K_{\sigma_n^s}$ only.\\
			Then, knowing $\left(K_{\sigma_m^s}\right)_m$, $\sigma_n^r-\sigma_n^s$ is the first hitting time of $k-\eps$ by the process $|W_\sigma|$. As $|W_{\sigma_n^s}|=\kappa$, by continuity, $W_\sigma$ keeps the same sign all along the interval $[\sigma_n^s,\sigma_n^r]$. 
			%		(Actually, if $W_{\sigma_n^s}=-\kappa$, by continuity, $W_\sigma$ will hitt $-\kappa+\eps<0$ before $0$ and will stay negative. We have the same if $W_{\sigma_n^s}=k$).
			Then we have:
			\begin{align*}
				\sigma_n^r-\sigma_n^s&=\inf\{\sigma>0 \  |W_{\sigma+\sigma_n^s}=\textrm{sign}(W_{\sigma_n^s})(\kappa-\eps)\}\\
				&=\inf\{\sigma>0 \ |(W_{\sigma+\sigma_n^s}-W_{\sigma_n^s})=\textrm{sign}(W_{\sigma_n^s})(\kappa-\eps-\kappa)\}\\
				&=\inf\left\{\sigma>0 \ \left\rvert\frac{\sqrt{k}R_{t_n^s}}{\sin\left(\frac{\sqrt{k}R_{t_n^s}}{2}\right)}(W_{\sigma+\sigma_n^s}-W_{\sigma_n^s})=-\frac{\sqrt{k}R_{t_n^s}}{\sin\left(\frac{\sqrt{k}R_{t_n^s}}{2}\right)}\textrm{sign}(W_{\sigma_n^s})\eps\right.\right\}.
			\end{align*}
			\\ Moreover, on $[\sigma_n^s,\sigma_n^r]$ and conditional on $K_{\sigma_n^s}$, $\frac{\sqrt{k}R_{t_n^s}}{\sin\left(\frac{\sqrt{k}R_{t_n^s}}{2}\right)}(W_{\sigma+\sigma_n^s}-W_{\sigma_n^s})_\sigma$ is a real Brownian motion starting from $0$.  
			Thus, conditional on $K_{\sigma_n^s}$, $\sigma_n^r-\sigma_n^s$ has the same law than $T_{a_n}$, the first hitting time of $a_n=\frac{\sqrt{k}R_{t_n^s}}{\sin\left(\frac{\sqrt{k}R_{t_n^s}}{2}\right)}\eps$ for a real Brownian motion starting at $0$. Using the Laplace transform we get: \begin{align*}&\esp[\exp(-\mu T_{a_n})]=\exp(-a_n\sqrt{2\mu}) \ \forall \mu>0\\
			\text{ and }&\esp[\exp\left(-\mu(\sigma_n^r-\sigma_n^s)\right)|K_{\sigma_n^s}]=\exp\left(-\frac{\sqrt{k}R_{t_n^s}}{\sin\left(\frac{\sqrt{k}R_{t_n^s}}{2}\right)}\eps\sqrt{2\mu}\right).\end{align*}
			Finally:
			\begin{equation*}\esp\left[\exp\left(-e^{2K_{\sigma_n^s}}(\sigma_n^r-\sigma_n^s)\right)|(K_{\sigma_m^s})_m\right]=\exp\left(-\frac{\sqrt{k}R_{t_n^s}}{\sin\left(\frac{\sqrt{k}R_{t_n^s}}{2}\right)}\eps\sqrt{2}e^{K_{\sigma_n^s}}\right).\end{equation*}
			Furthermore, conditional to $(K_{\sigma_m^s})_m$, $(\sigma_n^r-\sigma_n^s)_n$ are independent. Thus:
			\begin{equation*} 
				\esp\left[\exp\left(-\sum\limits_{n=0}^N e^{2K_{\sigma_n^s}}(\sigma_n^r-\sigma_n^s)\right)\bigg\rvert(K_{\sigma_m^s})_m\right]=\exp\left(-\sum\limits_{n=0}^N\frac{\frac{\sqrt{k}R_{t_n^s}}{2}}{\sin\left(\frac{\sqrt{k}R_{t_n^s}}{2}\right)}2\eps\sqrt{2}e^{K_{\sigma_n^s}}\right).
			\end{equation*}
			Using the dominated convergence theorem we get the announced equality.
			
			\item  We now deal with the quantity  $\displaystyle\sum\limits_{n\geq0} \frac{\frac{\sqrt{k}R_{t_n^s}}{2}}{\sin\left(\frac{\sqrt{k}R_{t_n^s}}{2}\right)}e^{K_{\sigma_n^s}}$ occurring in the previous result.\\
			Let us first notice that $\int_0^{\sigma_n^s}\Ne(\sigma)d\sigma=\sum\limits_{m=0}^{n-1}(\sigma_{m+1}^s-\sigma_m^r)$.
			During reflection times ($[\sigma_m^r,\sigma_{m+1}^s]$), the quantities $\sigma_{m+1}^s-\sigma_m^r$ are the first exit times from the open $]-\kappa,\kappa[$ of the diffusion $W_\sigma$ starting at $\pm(\kappa-\eps)\in]-\kappa,\kappa[$. Note that, unlike the fixed-distance coupling case, here the sign of $W_t$ can change.
			
			During these times we have:
			\begin{equation}\label{equaW}
	\left\lbrace
	\begin{aligned}
				d W_{\sigma(t)}\cdot dW_{\sigma(t)}&=\left(\frac{1}{4}\left(\frac{\tan\left(\frac{\sqrt{k}R_t}{2}\right)}{\frac{\sqrt{k}R_t}{2}}\right)^2+4 W_{\sigma(t)}^2\right)d\sigma(t)\\
				Drift(dW_{\sigma(t)})&= W_{\sigma(t)}\left(\frac{\sqrt{k}R_t}{2} \tan\left(\frac{\sqrt{k}R_t}{2}\right)+3\right) \ d\sigma(t).
			\end{aligned}
			\right.
\end{equation}
			We now take $m$ some positive integer. We define a new time-change: $\zeta_m(\sigma):=\int_0^\sigma d W_{s+\sigma_m^r}\cdot  dW_{s+\sigma_m^r}$.
			As $R_t$ is upper-bounded by $i(M)-\eta$ on $[0,\tau']$, there exists $M$ a positive constant such that $\frac{\tan\left(\frac{\sqrt{k}R_t}{2}\right)}{\frac{\sqrt{k}R_t}{2}}<M$ and $\frac{\sqrt{k}R_t}{2}\tan\left(\frac{\sqrt{k}R_t}{2}\right)<M$. Thus, \begin{equation*} \frac{\sigma}{4}\leq \zeta_m(\sigma)\leq \left(\frac{M^2}{4}+4\kappa ^2\right)\sigma \text{ and } \left|\frac{Drift(dW_{\sigma})}{d\zeta_m(\sigma)}\right|\leq 4\kappa(M+3).
			\end{equation*}
			
			Then there exists a one dimensional Brownian motion $B^m$, starting at $0$ and independent of $\zeta_m$, such that for all $\sigma\in[0,\sigma_{m+1}^s-\sigma_{m}^r]$: \begin{equation*}B_{\zeta_m(\sigma)}^m-4\kappa(3+M)\zeta_m(\sigma)\leq W_{\sigma+\sigma_m^r}-W_{\sigma_m^r}\leq B_{\zeta_m(\sigma)}^m+4\kappa(3+M)\zeta_m(\sigma).
			\end{equation*}
			We now obtain \begin{align*}\zeta_m(\sigma_{m+1}^s-\sigma_m^r)\geq \inf\{\zeta>0  \ |\ &B_{\zeta}^m+4\kappa(3+M)\zeta=\kappa-W_{\sigma_m^r}\}\\
			& \wedge \inf\{\zeta>0  \ |\ B_{\zeta}^m-4\kappa(3+M)\zeta=-\kappa-W_{\sigma_m^r}\}.\end{align*}
			As $W_{\sigma_m^r}$ can only take the two values $\kappa-\eps$ and $-(\kappa-\eps)$, we get $\zeta_m(\sigma_{m+1}^s-\sigma_m^r)\geq T_m$ with: \begin{equation*}T_m:=\inf\{\zeta>0  \ |\ B_{\zeta}^m+4\kappa(3+M)\zeta=\eps\}\wedge \inf\{\zeta>0  \ |\ B_{\zeta}^m-4\kappa(3+M)\zeta=-\eps\}.\end{equation*} In particular, we have: \begin{equation}\label{Tm}
\sigma_{m+1}^s-\sigma_m^r\geq\frac{1}{\frac{M^2}{4}+4\kappa ^2}T_m.
\end{equation} For all $m$, $(T_m)_m$ is a sequence of independent and equally distributed variables with non-negative and finite mean. Then, the strong law of large numbers yields:  \begin{equation*}
				\left(\frac{M^2}{4}+4\kappa ^2\right)\frac{1}{n}\sum\limits_{m=0}^{n-1}(\sigma_{m+1}^s-\sigma_m^r)\geq\frac{1}{n}\sum\limits_{m=0}^{n-1}T_m\xrightarrow[n\rightarrow +\infty]{a.s.}\esp[T_0].
			\end{equation*}
			
			Thus, for $n$ large enough, we get a.s.: \begin{equation}\label{lfgn 2}
				\frac{\int_0^{\sigma_n^s}\Ne(\sigma)d\sigma}{n}=\frac{1}{n}\sum\limits_{m=0}^{n-1}(\sigma_{m+1}^s-\sigma_m^r)\geq \frac{\esp[T_0]}{2(\frac{M^2}{4}+4\kappa ^2)}>0.
			\end{equation}
			Moreover, we obtain  $\int_0^{\sigma_n^s}\Ne(\sigma)d\sigma\xrightarrow[n\rightarrow +\infty]{a.s}+\infty$.\\
			Let us now recall that $K_\sigma=K_0+\int_0^\sigma\Ne(s) dC_s-\int_0^\sigma \Ne\left(\frac{\sqrt{k}R_t \tan\left(\frac{\sqrt{k}R_t}{2}\right)}{4}+\frac{1}{2}\right)ds$. Thus  $K_\sigma-K_0+\int_0^\sigma \Ne\left(\frac{\sqrt{k}R_t \tan\left(\frac{\sqrt{k}R_t}{2}\right)}{4}+\frac{1}{2}\right)ds$ is a Brownian motion for the change of time $\int_0^\sigma\Ne(s)ds$.\\
			%		Moreover, $\int_0^S\Ne(\sigma)d\sigma=+\infty$. Actually, $\int_0^S\Ne(\sigma)d\sigma=\int_0^\tau \Ne(t)\times\frac{4}{R_t^2}dt$ has the same law than $\int_0^T \frac{4}{4(\frac{R_0}{2}+\hat{C}_t)^2}dt$ where $\hat{C}$ is a real Brownian motion starting at $0$ and $T$ is it's first hitting time of $-\frac{R_0}{2}$. Let notice that $d(\ln(\frac{R_0}{2}+\hat{C}_t))=\frac{d\hat{C}_t}{\frac{R_0}{2}+\hat{C}_t}-\frac{dt}{4(\frac{R_0}{2}+\hat{C}_t)^2}$ thanks to the Itô's formula. Thus:
			%		\begin{align*} 
			%		\esp[\int_0^S\Ne(\sigma)d\sigma]&=\esp[\int_0^T \frac{4}{4(\frac{R_0}{2}+\hat{C}_t)^2}dt]\\
			%		&=\esp[4\int_0^T (-d(\ln(\frac{R_0}{2}+\hat{C}_t))+\frac{d\hat{C}_t}{\frac{R_0}{2}+\hat{C}_t})]\\
			%		&=\esp[-4(\ln(\frac{R_0}{2}+\hat{C}_T)-\ln(\frac{R_0}{2}+\hat{C}_0)]\\
			%		&=+\infty \text{ since } \hat{C}_T=-\frac{R_0}{2}
			%		\end{align*}
			%		And we get $\int_0^S\Ne(\sigma)d\sigma=+\infty$ a.s..\\
			
			Note that: \begin{equation*}\frac{1}{2}\leq \frac{\int_0^{\sigma_n^s} \Ne\left(\frac{\sqrt{k}R_t \tan\left(\frac{\sqrt{k}R_t}{2}\right)}{4}+\frac{1}{2}\right)ds}{\int_0^{\sigma_n^s}\Ne(s)ds}.\end{equation*}
			By the strong law of large number for Brownian motions, we also have: \begin{equation*}\frac{K_{\sigma_n^s}-K_0+\int_0^{\sigma_n^s} \Ne\left(\frac{\sqrt{k}R_t \tan\left(\frac{\sqrt{k}R_t}{2}\right)}{4}+\frac{1}{2}\right)ds}{\int_0^{\sigma_n^s}\Ne(s)ds} \xrightarrow[n\rightarrow +\infty]{a.s}0.\end{equation*}
			Finally, a.s. for $n$ large enough, we obtain:
			\begin{equation} \label{lfgn 1}
				\frac{K_{\sigma_n ^s}}{\int_0^{\sigma_n ^s}\Ne(s)ds}\leq-\frac{1}{4}.
			\end{equation}

			By combining the results (\ref{lfgn 2}) and  (\ref{lfgn 1}), we get a.s. for $n$ large enough:
			\begin{equation*}
				\frac{1}{n}K_{\sigma_n^s}=\frac{K_{\sigma_n^s}}{\int_0^{\sigma_n^s}\Ne(s)ds}\times \frac{\int_0^{\sigma_n^s}\Ne(s)ds}{n}\leq -\frac{c_0}{4} <0
			\end{equation*}
			with $c_0:= \frac{\esp[T_0]}{2\left(\frac{M^2}{4}+4\kappa ^2\right)}$.
			It remains to notice that a.s., for all $n$, we have $0<\frac{\sqrt{k}R_{t_n^s}}{2}<\frac{\pi}{2}$. Thus: \begin{equation*}\frac{\frac{\sqrt{k}R_{t_n^s}}{2}}{\sin\left(\frac{\sqrt{k}R_{t_n^s}}{2}\right)}\leq \frac{\pi}{2}\text{ and }\sum\limits_{n\geq M}\frac{\frac{\sqrt{k}R_{t_n^s}}{2}}{\sin\left(\frac{\sqrt{k}R_{t_n^s}}{2}\right)}e^{K_{\sigma_n^s}}\leq\frac{\pi}{2}\sum\limits_{n\geq M}(e^{-\frac{c_0}{4}})^n<+\infty.\end{equation*}
		 Finally we get: \begin{equation}\label{sommefinie}
				\sum\limits_{n\geq 0}\frac{\frac{\sqrt{k}R_{t_n^s}}{2}}{\sin\left(\frac{\sqrt{k}R_{t_n^s}}{2}\right)}e^{K_{\sigma_n^s}}<+\infty \ a.s..
			\end{equation}
			\item Using (\ref{esp1}) and (\ref{sommefinie}), we have $\esp[\exp(-\sum\limits_{n\geq 0}e^{K_{\sigma_n^s}}(\sigma_n^r-\sigma_n^s))|(K_{\sigma_m^s})_m]>0$ a.s..\\
			Thus, still conditional to $(K_{\sigma_m^s})_m$, 
			%since $\exp(-\sum\limits_{n\geq 0}e^{K_{\sigma_n^s}}(\sigma_n^r-\sigma_n^s))\geq 0$,
			the event $\exp(-\sum\limits_{n\geq 0}e^{K_{\sigma_n^s}}(\sigma_n^r-\sigma_n^s))> 0$ has a non-zero probability. Equivalently the event $\sum\limits_{n\geq 0}e^{K_{\sigma_n^s}}(\sigma_n^r-\sigma_n^s)<+\infty$ has a non-zero probability. As $\left(e^{2K_{\sigma_n^s}}(\sigma_n^r-\sigma_n^s)\right)_n$ are independent, using the Kolmogorov zero-one law, we get $\pr(\sum\limits_{n\geq 0}e^{K_{\sigma_n^s}}(\sigma_n^r-\sigma_n^s)<+\infty|(K_{\sigma_m^s})_m)=1$ a.s..
			\\ 
			Finally, $\sum\limits_{m\geq 0}e^{2K_{\sigma_m^s}}(\sigma_m^r-\sigma_m^s)<+\infty$ a.s. and so $\tau'<+\infty$ a.s..
		\end{itemize}
		To show that $\pr(\tau>\tau_{\eta})<1$, we just need to remark that this event only depends on the evolution of $R_t$. As $R_t$ is acting as a time-changed Brownian motion with negative drift omitting the times where it stays constant, we directly obtain our result.
	\end{proof}
	We can now give the proof of Theorem \ref{successful}.
	\begin{proof}[Proof of Theorem \ref{thm: couplingAreaInR}]
	To construct a successful coupling, we just need to start the coupling described in Theorem \ref{Kendall}.
	\begin{enumerate}[label=(\roman*)]
		\item If $\tau\wedge \tau_{\eta}=\tau$, we obtain $R_t=A_t=0$.
		\item If $\tau\wedge \tau_{\eta}=\tau_{\eta}$, we use a synchronous coupling until $R_t=R_0$ and then a fixed distance coupling until $A_t=0$ and we re-start the coupling of Theorem \ref{Kendall}.
	\end{enumerate}
	At step (ii), as $R_t$ is deterministic and decreasing during synchronous coupling, each "synchronous step" will take a constant finite time. It also takes an a.s. finite time to obtain $A_t=0$ with fixed distance coupling. Thus we repeat the same experiment independently. As the probability that $\tau<\tau_{\eta}$ is non-zero during the coupling of Theorem \ref{Kendall}, it will take a finite number of change of coupling to be in this event and then have an a.s. finite time of success.
	Let us remind that, if $R_0\in\{0,i(M)\}$, we need to use coupling of independent Brownian motions on any finite deterministic interval of time before our successful coupling to start with $R_t\in]0,i(M)[$. This does not change our result.
	\end{proof}
\begin{NB}\label{NB: compactness}
Note that, using the compactness of $SU(2)$, we can make two improvements for this strategy to define a successful coupling in $SU(2)$. We take $k=1$.
\begin{itemize}
    \item Considering $A_t \mod (4\pi)$ instead of a value in $\mathbb{R}$, we can suppose that $A_t$ stays in $]-2\pi,2\pi]$. The idea is to take $\kappa$ small enough such that 
%This time we can take $\kappa<\frac{2}{\pi}$. This way, 
fixed-distance coupling is stopped when the process $W_{\sigma}$ starting at $\pm \kappa$ hits $\pm(\kappa-\eps)$ with this time a non zero probability to reach one or another of these values. Then there will be less time spent in fixed-distance coupling than in Theorem \ref{Kendall}. To be more precise,  for $n\geq 0$, there exists $(V_{\sigma})_{\sigma}$ a Brownian motion starting at $0$ such that:
\begin{align*}
    \sigma_n^r-\sigma_n^s &=\inf \Bigg\{\sigma \ \Big\rvert \ \sgn(W_{\sigma_n^s})\kappa+\frac{\sin\left(\frac{R_{t_n^s}}{2}\right)}{R_{t_n^s}}V_{\sigma}\in\Big\{\sgn(W_{\sigma_n^s})(\kappa-\eps) \\
    & \ ;\sgn(W_{\sigma_n^s})\left(\frac{4\pi}{R_{t_n^s}^2}-(\kappa-\eps)\right)\Big\}\Bigg\}
    \\
    &=\inf \big\{\sigma \ | \ \sgn(W_{\sigma_n^s})\frac{\sin\left(\frac{R_{t_n^s}}{2}\right)}{R_{t_n^s}}V_{\sigma}\in\{-\eps ; \frac{4\pi}{R_{t_n^s}^2}-2\kappa+\eps\}\big\}.
\end{align*}
 If $\frac{4\pi}{R_{t_n^s}^2}-2\kappa+\eps>0$, which is the case for $\kappa<\frac{2}{\pi}+\frac{\eps}{2}$, then the above stopping time is less than $T_{a_n}$ with $a_n$ defined as before.
This way, instead of (\ref{esp1}), we obtain:
\begin{equation*}
				\esp\left[\exp\left(\left.-\sum\limits_{n\geq 0}e^{2K_{\sigma_n^s}}(\sigma_n^r-\sigma_n^s)\right)\right|(K_{\sigma_m^s})_m\right]>\exp\left(-\sum\limits_{n\geq 0}e^{K_{\sigma_n^s}}\times\frac{\frac{R_{t_n^s}}{2}}{\sin\left(\frac{R_{t_n^s}}{2}\right)}\times 2\sqrt{2}\eps\right).
			\end{equation*}
\item In addition to decreasing the time spent in fixed-distance coupling, having bounded values for $A_t$ prevents from using synchronous coupling when $R_t$ is too close of $i(M)=\pi$. Let us explain this.\\
During reflection coupling, $W_{\sigma}$ stays continuous and still satisfies (\ref{equaW}). Supposing that $R_t\leq \pi-\eta$ for all $t$ in $[t_m^r,t_{m+1}^s]$, we still have (\ref{Tm}).
Moreover, supposing that $R_t>\pi-\eta$ for some $t\in[t_m^r,t_{m+1}^s]$, we get $|W_{\sigma(t)}|=\frac{|A_t|}{R_t^2}\leq \frac{2\pi}{(\pi-\eta)^2}$. Choosing $\eta$, $\kappa$ and $\delta$ such that $\frac{2\pi}{(\pi-\eta)^2}<\kappa-\delta<2\pi$, we obtain $|W_{\sigma(t)}|< \kappa$. This is the case for $\kappa-\delta>\frac{2}{\pi}$. Note that, in order to keep the previous improvement, we take $\delta<\frac{\eps}{2}$. Thus reflection coupling will not end while $R_t$ stays up to $\pi-\eta$. Moreover, there will exist $\tilde{t}_m^r\in]t_m^r,t_{m+1}^s[$ such that $|W_{\sigma(\tilde{t}_m^r)}|= \kappa-\delta$ and $R_t\leq \pi-\eta$ for all $t$ in $[\tilde{t}_m^r,t_{m+1}^s]$. As before, we obtain: \begin{equation*}
\sigma_{m+1}^s-\sigma_m^r>\sigma_{m+1}^s-\sigma(\tilde{t}_m^r)\geq\frac{1}{\frac{M^2}{4}+4\kappa ^2}\tilde{T_m}\end{equation*}
with, $\tilde{T}_m$ defined as $T_m$ but with $\delta$ instead of $\eps$: \begin{equation*}\tilde{T}_m:=\inf\{\zeta>0  \ |\ B_{\zeta}^m+4\kappa(3+M)\zeta=\delta\}\wedge \inf\{\zeta>0  \ |\ B_{\zeta}^m-4\kappa(3+M)\zeta=-\delta\}.\end{equation*}

Thus, using $(\tilde{T}_m)_m$ instead of $(T_m)_m$, we get again (\ref{lfgn 2}). Following the rest of the proof of Theorem \ref{Kendall}, we directly obtain $\tau<+\infty$ a.s..
\end{itemize}
\end{NB}
\begin{appendices}
	\section{Computation of the subLaplacian on SU(2)}\label{SubLapSU(2)}
	The aim of this appendix is to provide a proof for the expression (\ref{SubLapfSU}) of the subLaplacian operator of $SU(2)$ in cylindrical coordinates. Contrary to the other proofs previously mentioned, this proof is only using Lie brackets relations (\ref{LieSU}) and is easy to adapt to $SL(2,\mathbb{R})$ as we will show in \ref{SubLapSL(2)}. \\
	The first step is the computation of the left invariant vectors $\bar{X}$ and $\bar{Y}$ in cylindrical coordinates. In order to do that, we need to express the following expressions in cylindrical coordinates: \begin{align*}&\exp\big(\ph(\cos(\theta)X+\sin(\theta)Y)\big)\exp(zZ)\exp(\eps X)\\
	\text{ and }&\exp\big(\ph(\cos(\theta)X+\sin(\theta)Y)\big)\exp(zZ)\exp(\eps Y).\end{align*} We will then study the derivation of these coordinates in $\eps=0$.
	
	The triplet $(X,Y,Z)$ is a basis of the Lie Algebra $\mathfrak{su}(2)$ satisfying (\ref{LieSU}).
	Thus it can be seen as a sort of direct orthonormal basis for the action of rotation $[\cdot,\cdot]$ which we will denote $\cdot\wedge\cdot$ as the tensor operator, in order to simplify the computations.	In the following lemmas we will deal with some basis $(u,v,w)$ of the Lie algebra satisfying (\ref{LieSU}), that is, such that $u\wedge v=w$, $v\wedge w=u$ and $w\wedge u=v$. We will focus on expressions of the type \begin{equation*}\exp\left(\alpha u+\beta v+\gamma w\right)\exp\left(\epsilon\left(\alpha' u+\beta' v+\gamma' w\right)\right)\end{equation*} for $\alpha, \beta,\gamma,\alpha' ,\beta',\gamma'$ real numbers and $\epsilon$ real near $0$.
	\begin{lemme}\label{orth}
		Let $(u,v,w)$ be a basis satisfying (\ref{LieSU}). Then we get:
		
		\begin{equation}\label{a}\exp\left(\alpha u\right)\exp\left(\epsilon\beta v\right)=\exp\left(\alpha u+\epsilon\frac{\beta\alpha}{2}\left(\cot\left(\frac{\alpha}{2}\right)v+w\right)\right)+\mathcal{O}\left(\epsilon^2\right).\end{equation}
		
		%\begin{equation}\label{b}\exp(\epsilon\beta v)\exp(\alpha u)=\exp(\alpha u+\epsilon\frac{\beta\alpha}{2}(\cot(\frac{\alpha}{2})v-w))+\mathcal{O}(\epsilon^2)	\end{equation}		

		More generally we have: 
		\begin{equation}\label{c}
		\exp(\alpha u)\exp(\beta v)=\exp\left(\beta(\cos\left(\alpha\right)v+\sin\left(\alpha\right)w)\right)\exp\left(\alpha u\right).
		\end{equation}
	\end{lemme}
	\begin{proof}
		Let us deal with (\ref{a}). Using the Campbell Hausdorff formula we have \begin{equation*}\exp(\alpha u)\exp(\epsilon \beta v)=\exp\left(\alpha u+\eps\beta\psi(-\ad_{\alpha u})(v)\right)+\mathcal{O}(\eps^2)\end{equation*}
		with $\psi(z)=\frac{z}{e^z-1}=\sum\limits_{n\geq 0}\frac{B_n}{n!}z^n$, $B_n$ denoting the Bernoulli coefficients.\\
		We have: $\ad_{(\alpha u)}(v)=\alpha u\wedge v=\alpha w$ and $\ad_{(\alpha u)}^{(2)}(v)=\alpha^2 u\wedge w=-\alpha^2 v$.\\ So we obtain:
		\begin{equation*}\ad_{(\alpha u)}^{(k)}(v)=\left\{
		\begin{array}{ll}
			(-1)^n\alpha^{2n}v & \text{if } k=2n,\ n\geq 0;\\
			(-1)^n\alpha^{2n+1}w & \text{if } k=2n+1,\ n\geq 0.
		\end{array}\right. \end{equation*}
		Thus:
		\begin{align*}\psi\left(-\ad_{ (\alpha u)}\right)&=\sum\limits_{n\geq 0}\frac{B_{2n}}{(2n)!}(-1)^n\alpha^{2n}v-\sum\limits_{n\geq 0}\frac{B_{2n+1}}{(2n+1)!}(-1)^n\alpha^{2n+1}w\\
		&=\textit{Re}\left(\psi(i\alpha)\right)v-\textit{Im}\left(\psi(i\alpha)\right)w.\end{align*}
		To obtain the announced equality, we just need to note that:
		\begin{equation*}
		    \psi(i\alpha)=\frac{i\alpha}{e^{i\alpha}-1}=\frac{i\alpha e^{-i\frac{\alpha}{2}}}{e^{i\frac{\alpha}{2}}-e^{-i\frac{\alpha}{2}}}=\frac{\alpha e^{-i\frac{\alpha}{2}}}{2\sin(\frac{\alpha}{2})}=\frac{\alpha}{2}\left(\cot\left(\frac{\alpha}{2}\right)-i\right).
		\end{equation*}
		%We can then obtain (\ref{b}):
		%\begin{align*}
		%\exp(\epsilon\beta v)\exp(\alpha u)=&(\exp(-\alpha u)\exp(-\epsilon \beta v))^{-1}\\
		%=&(\exp(-\alpha u+\epsilon\frac{\beta\alpha}{2}(\cot(\frac{-\alpha}{2})v+w))+\mathcal{O}(\epsilon^2))^{-1} \text{using (\ref{a})};\\
		%=&\exp(\alpha u-\epsilon\frac{\beta\alpha}{2}(\cot(\frac{-\alpha}{2})v+w))+\mathcal{O}(\epsilon^2)\\
		%=&\exp(\alpha u+\epsilon\frac{\beta\alpha}{2}(\cot(\frac{\alpha}{2})v-w))+\mathcal{O}(\epsilon^2)
		%\end{align*}\linebreak
		%We could use (\ref{a}) and (\ref{b}) to obtain directly $\exp(\alpha u)\exp(\epsilon\beta v)=\exp(\epsilon\beta(\cos(\alpha)v+\sin(\alpha)w))\exp(\alpha u)+\mathcal{O}(\eps^2)$ by identification. However, as (\ref{c}) is more general and easy to compute, we will deal with it.\\
		Notice that $\exp(A)\exp(B)\exp(-A)=\exp\left(\mbox{Ad}\left(\exp(A)\right)(B)\right)=\exp\left(e^{\ad_ A}(B)\right)$ for any matrix $A$, $B$ with real or complex coefficients. Thus:\begin{align*}\exp(\alpha u)\exp(\beta v)&=\exp(\alpha u)\exp(\beta v)\exp(-\alpha u)\exp(\alpha u)\\
		&=\exp\left(e^{\ad_ {\alpha u}}\left(\beta v\right)\right)\exp(\alpha u).\end{align*}
		We finally obtain (\ref{c}) as: \begin{align*}e^{\ad_ {\alpha u}}(\beta v)&=\sum\limits_{k\geq 0}\frac{1}{k!}\ad_{\alpha u}^{(k)}(\beta v)
			=\beta\left(\sum\limits_{n\geq 0}\frac{\alpha^{2n}}{(2n)!}(-1)^n v+\sum\limits_{n\geq 0}\frac{\alpha^{2n+1}}{(2n+1)!}(-1)^n w\right)\\
			&=\beta\left(\cos(\alpha)v+\sin(\alpha)w\right).\end{align*}
	\end{proof}
	\begin{lemme}\label{north}
		Let $(u,v,w)$ be a basis of our Lie algebra satisfying (\ref{LieSU}) and $t$ such that $u\wedge t=\rho w$; $\rho\neq 0$. Then we get:
		\begin{equation}
			\exp(\alpha u)\exp(\eps t)=\exp\left(\alpha u+\eps \left(t+\rho \left(-v+\frac{\alpha}{2}\cot\left(\frac{\alpha}{2}\right)v+\frac{\alpha}{2}w\right)\right)\right)+\mathcal{O}(\eps^2).
		\end{equation}
	\end{lemme}
	\begin{proof}
		We use exactly the same strategy as for the previous case using the Campbell Hausdorff formula.
		This time we have: $\ad_{(\alpha u)}(t)=\alpha u\wedge t=\alpha\rho w$; $\ad_{(\alpha u)}^{(2)}(t)=-\alpha^2\rho v$; $\ad_{(\alpha u)}^{(3)}(v)=-\alpha^3\rho w$.\\ Thus we obtain: 
		\begin{equation*}\ad_{(\alpha u)}^{(k)}(t)=\left\{
	    \begin{array}{ll}
			t & \text{if }k=0\\
			(-1)^n\alpha^{2n}\rho v & \text{if } k=2n, n\geq 1 \\
			(-1)^n\alpha^{2n+1}\rho w & \text{if } k=2n+1,\ n\geq 0
		\end{array}\right. \end{equation*}
	and
	\begin{align*}\psi\left(-\ad_{\alpha u}\right)(t)&=t+\rho\left(\text{Re}\left(\psi(i\alpha)\right)-1\right)v-\rho\text{Im}\left(\psi(i\alpha)\right)w\\
	&=t+\rho\left(\frac{\alpha}{2}\cot\left(\frac{\alpha}{2}\right)-1\right)v+\rho\frac{\alpha}{2}w.\end{align*}
	\end{proof}
	\begin{NB}
		Note that the two lemmas stay true if the coefficients $\alpha$ and $\beta$ are complex.
	\end{NB}
	Now we can find the expressions of the invariant vectors on $SU(2)$ in cylindrical coordinates. We take $(\ph,\theta,z)$ some cylindrical coordinates of an element of $SU(2)$. We will denote $x=\ph\cos(\theta)$ and $y=\ph\sin(\theta)$. We look for the cylindrical coordinates of $\exp(xX+yY)\exp(zZ)\exp(\eps(aX+bY))$, with $a$,$b\in\mathbb{R}$ and $\eps$ small.
	\begin{itemize}
		\item First, we use Lemma \ref{orth} with $u=Z$, $v=\frac{aX+bY}{r}$, $\alpha=z$, $\beta=r$ and $r=\sqrt{a^2+b^2}$. Note that $w=\frac{-bX+aY}{r}$. we get:
		\begin{align*}
			\exp(zZ)\exp(\eps(aX+bY))&=\exp(\eps r(\cos(z)v+\sin(z)w))\exp(zZ)\\
			&=\exp(\eps (\cos(z)(aX+bY)+\sin(z)(-bX+aY)))\exp(zZ)\\
			&=\exp(\eps(A(z)X+B(z)Y))\exp(zZ)\\
			\text{ with } A(z)=a\cos(z)-b&\sin(z) \text{ and }B(z)=a \sin(z)+b\cos(z).
		\end{align*}
		If $x=y=0$, we obtain the expected expression in cylindrical coordinates. For what will follow, we will suppose that $(x,y)\neq(0,0)$. We recall that $\ph=\sqrt{x^2+y^2}$.
		\item To obtain the expression of $\exp(xX+yY)\exp(\eps(A(z)X+B(z)Y))$ in cylindrical coordinates, we use Lemma \ref{north} with $u=\frac{xX+yY}{\ph}$, $\alpha=\ph$ and $t=A(z)X+B(z)Y$. We have $u\wedge t=\frac{xB(z)-yA(z)}{\ph}Z$ thus we take $w=Z$ and $\rho=\frac{xB(z)-yA(z)}{\ph}$. Taking $v=w\wedge u=\frac{-yX+xY}{\ph}$, we obtain a basis $(u,v,w)$ satisfying (\ref{LieSU}). We get:
		\begin{align*}
			\exp\left(xX+yY\right)\exp\left(\eps\left(A(z)X+B(z)Y\right)\right)&=\exp\bigg{(}xX+yY+\eps \bigg{(}A(z)X+B(z)Y+\\\rho \bigg{(}-\frac{-yX+xY}{\ph}+\frac{\ph}{2}&\cot\left(\frac{\ph}{2}\right)\frac{-yX+xY}{\ph}+\frac{\ph}{2}Z\bigg{)}\bigg{)}\bigg{)}+\mathcal{O}(\eps^2)\\
			&=\exp\left(\left(x+\eps C\right)X+\left(y+\eps D\right)Y+\eps E\cdot Z\right)
		\end{align*}
		with $C=A(z)+\rho y\left(\frac{1}{\ph}-\frac{1}{2}\cot\left(\frac{\ph}{2}\right)\right)$, 
		$D=B(z)+x\rho\left(-\frac{1}{\ph}+\frac{1}{2}\cot\left(\frac{\ph}{2}\right)\right)$, 
		$E=\rho\frac{\ph}{2}$.
		\item In cylindrical coordinates we can write:
		\begin{equation*}\exp\left(\left(x+\eps C\right)X+\left(y+\eps D\right)Y+\eps E Z\right)=\exp\left(x(\eps) X+y(\eps)Y\right)\exp\left(\eps \gamma(\eps)Z\right).\end{equation*}
		We use the first equality of Lemma \ref{orth} with $u=\frac{x(\eps)X+y(\eps)Y}{\ph(\eps)}$, $\alpha=\ph(\eps)=\sqrt{x(\eps)^2+y(\eps)^2}$, $v=Z$ and $\beta=\gamma(\eps)$. We also have $w=\frac{y(\eps)X-x(\eps)Y}{\ph(\eps)}$. Then we obtain:
		\begin{align*}
			&\exp\left(x(\eps) X+y(\eps)Y\right)\exp\left(\eps \gamma(\eps)Z\right)\\
			&=\exp \bigg{(}x(\eps)X+y(\eps)Y+\eps\frac{\gamma(\eps)\ph(\eps)}{2} \bigg{(}\cot\left(\frac{\ph(\eps)}{2}\right)Z+\frac{y(\eps)X-x(\eps)Y}{\ph(\eps)} \bigg{)}\bigg{)}+\mathcal{O}(\epsilon^2)\\
			&=\exp\bigg{(}\left(x(\eps)+\eps\frac{\gamma(\eps)}{2}y(\eps)\right)X+\left(y(\eps)-\eps\frac{\gamma(\eps)}{2}x(\eps)\right)Y+\eps\frac{\gamma(\eps)\ph(\eps)}{2}\cot\left(\frac{\ph(\eps)}{2}\right)Z\bigg{)}\\&+\mathcal{O}(\epsilon^2).
		\end{align*}
		Thus we take $x(\eps), y(\eps)$ and $\gamma(\eps)$ such that:
		\begin{subequations}
  \begin{empheq}[left=\empheqlbrace]{align}
    x+\eps C &=x(\eps)+\eps\frac{\gamma(\eps)}{2}y(\eps)\label{eq a} \\ 
    y+\eps D &=y(\eps)-\eps\frac{\gamma(\eps)}{2}x(\eps)\label{eq b} \\
    E &=\frac{\ph(\eps)\gamma(\eps)}{2}\cot(\frac{\ph(\eps)}{2})\label{eq c}
  \end{empheq}
\end{subequations}
		Note that $x(0)=x$, $y(0)=y$ and $\ph(0)=\ph$.
		Finally we get: \begin{align*}\exp\left(xX+yY  \right)\exp\left(zZ\right)&\exp\left(\eps(aX+bY)\right)\\
		&=\exp\left(x(\eps) X+y(\eps)Y\right)\exp\left(\eps \gamma(\eps)Z\right)\exp\left(zZ\right)+\mathcal{O}(\eps^2)\\
			&=\exp\left(x(\eps) X+y(\eps)Y\right)\exp\left(z(\eps)Z\right)+\mathcal{O}(\eps^2)\\
			&\text{ with } z(\eps)=z+\eps \gamma(\eps)\\
			&=\exp\left(\ph(\eps)(\cos(\theta(\eps)) X+\sin(\theta(\eps))Y\right)\exp\left(z(\eps)Z\right)+\mathcal{O}(\eps^2).\end{align*}
		\item To obtain the values of the left invariant vector, we now have to study $\theta'(0)$, $\ph'(0)$ and $z'(0)$.\\
		Using (\ref{eq c}), we directly have $z'(0)=\gamma(0)=\frac{2}{\ph}\tan\left(\frac{\ph}{2}\right)E=\rho \times \tan\left(\frac{\ph}{2}\right)$ with \begin{align*} \rho&=\frac{xB(z)-yA(z)}{\ph}=\cos(\theta)B(z)-\sin(\theta)A(z)\\
				&=\cos(\theta)\left(a \sin(z)+b\cos(z)\right)-\sin(\theta)\left(a\cos(z)-b\sin(z)\right)\\
				&=a\left(\cos(\theta)\sin(z)-\sin(\theta)\cos(z)\right)+b\left(\cos(\theta)\cos(z)+\sin(\theta)\sin(z)\right)\\
				&=a\sin(z-\theta)+b\cos(z-\theta).
			\end{align*}
			Thus $z'(0)=\tan\left(\frac{\ph}{2}\right)\left(a\sin(z-\theta)+b\cos(z-\theta)\right)$.\\
		Then, summing the squares of (\ref{eq a}) and (\ref{eq b}), we get $\ph(\eps)^2=x(\eps)^2+y(\eps)^2=x^2+y^2+2\eps\left(xC+yD\right)+\mathcal{O}(\eps^2)$, hence $2\ph(0)\ph'(0)=2\left(xC+yD\right)$ and \begin{align*}
				\ph'(0)&=\frac{xC+yD}{\ph}=\cos(\theta)C+\sin(\theta)D\\
				&=\cos(\theta)\left(A(z)+\rho y\left(\frac{1}{\ph}-\frac{1}{2}\cot\left(\frac{\ph}{2}\right)\right)\right)\\
				&+\sin(\theta)\left(B(z)+x\rho\left(-\frac{1}{\ph}+\frac{1}{2}\cot\left(\frac{\ph}{2}\right)\right)\right)\\
				&=\cos(\theta)A(z)+\sin(\theta)B(z)+\rho \left(\frac{1}{\ph}-\frac{1}{2}\cot\left(\frac{\ph}{2}\right)\right)\left(y\cos(\theta)-x\sin(\theta)\right)\\
				&=\cos(\theta)A(z)+\sin(\theta)B(z)\\
				&=\cos(\theta)\left(a\cos(z)-b\sin(z)\right)+\sin(\theta)\left(a \sin(z)+b\cos(z)\right)\\
				%&=a\left(\cos(\theta)\cos(z)+\sin(\theta)\sin(z)\right)+b\left(\sin(\theta)\cos(z)-\cos(\theta)\sin(z)\right)\\
				&=a\cos(\theta-z)+b\sin(\theta-z).
			\end{align*}
			
		Now note that $\tan\left(\theta(\eps)\right)=\frac{y(\eps)}{x(\eps)}$.\\
			Thus $\frac{\theta'(\eps)}{\cos^2(\theta(\eps))}=\frac{y'(\eps)x(\eps)-x'(\eps)y(\eps)}{x(\eps)^2}$ and $\theta'(0)=\cos^2(\theta)\frac{y'(0)x-x'(0)y}{r^2 \cos^2(\theta)}=\frac{y'(0)x-x'(0)y}{r^2}$.\\ Using (\ref{eq a}) and then (\ref{eq b}), we get: \begin{align*}x'(0)=C-\frac{\gamma(0)}{2}y(0)=C-\frac{\rho}{2} \tan(\frac{\ph}{2})y\\  y'(0)=D+\frac{\gamma(0)}{2}x(0)=D+\frac{\rho}{2} \tan(\frac{\ph}{2})x.
			\end{align*}
			We obtain: \begin{align*}\theta'(0)
				&=\frac{\left(D+\frac{\rho}{2} \tan\left(\frac{\ph}{2}\right)x\right)x-\left(C-\frac{\rho}{2} \tan\left(\frac{\ph}{2}\right)y\right)y}{\ph^2}=\frac{Dx-Cy}{\ph^2}+\frac{\rho}{2} \tan\left(\frac{\ph}{2}\right)\\
				&=\frac{D\cos(\theta)-C\sin(\theta)}{\ph}+\frac{\rho}{2} \tan\left(\frac{\ph}{2}\right).
			\end{align*}
			We have: \begin{align*}&D\cos(\theta)-C\sin(\theta)\\&=\cos(\theta)\left(B(z)+x\rho\left(-\frac{1}{\ph}+\frac{1}{2}\cot\left(\frac{\ph}{2}\right)\right)\right)-\sin(\theta)\left(A(z)+\rho y\left(\frac{1}{\ph}-\frac{1}{2}\cot\left(\frac{\ph}{2}\right)\right)\right)\\
				&=\cos(\theta)B(z)-\sin(\theta)A(z)+\rho\left(-\frac{1}{\ph}+\frac{1}{2}\cot\left(\frac{\ph}{2}\right)\right)\left(x\cos(\theta)+y\sin(\theta)\right)\\
				&=\rho+\ph\rho\left(-\frac{1}{\ph}+\frac{1}{2}\cot\left(\frac{\ph}{2}\right)\right)\\
				&=\frac{\ph\rho}{2}\cot\left(\frac{\ph}{2}\right).
			\end{align*}
			Hence: $\theta'(0)=\frac{\rho}{2}\left(\cot\left(\frac{\ph}{2}\right)+\tan\left(\frac{\ph}{2}\right)\right)=\frac{a\sin(z-\theta)+b\cos(z-\theta)}{2}\left(\cot\left(\frac{\ph}{2}\right)+\tan\left(\frac{\ph}{2}\right)\right).$
		\item Finally, taking $(a,b)=(1,0)$, we get $\bar{X}=\begin{pmatrix} \cos(\theta-z) \\\frac{1}{2}\sin(z-\theta)\left(\cot\left(\frac{\ph}{2}\right)+\tan\left(\frac{\ph}{2}\right)\right) \\\tan\left(\frac{\ph}{2}\right)\sin(z-\theta) \end{pmatrix}$ and, taking $(a,b)=(0,1)$, we get $\bar{Y}=\begin{pmatrix} \sin(\theta-z) \\\frac{1}{2}\cos(\theta-z)\left(\cot\left(\frac{\ph}{2}\right)+\tan\left(\frac{\ph}{2}\right)\right) \\\tan\left(\frac{\ph}{2}\right)\cos(z-\theta) \end{pmatrix}$
		.\end{itemize}
	We can now calculate the subLaplacian operator $L=\frac{1}{2}\left(\bar{X}^2+\bar{Y}^2\right)$ (also made in~\cite{bonnefont-these}).
	\begin{align*}
	\bar{X}^2=\cos(\theta-z)\partial_\ph&\left[\cos(\theta-z)\partial_\ph+\frac{1}{2}\sin(z-\theta)\left(\cot\left(\frac{\ph}{2}\right)+\tan\left(\frac{\ph}{2}\right)\right)\partial_{\theta}\right.
	\\&+\left.\tan\left(\frac{\ph}{2}\right)\sin(z-\theta)\partial_z\right] \\
		+\frac{1}{2}\sin(z-\theta)&\left(\cot\left(\frac{\ph}{2}\right)+\tan\left(\frac{\ph}{2}\right)\right)\partial_{\theta}\bigg[\cos(\theta-z)\partial_\ph
		+\frac{1}{2}\sin(z-\theta)\left(\cot\left(\frac{\ph}{2}\right)\right.\\
		&\left.+\tan\left(\frac{\ph}{2}\right)\right)\partial_{\theta}+\tan\left(\frac{\ph}{2}\right)\sin(z-\theta)\partial_z\bigg]\\
		+\tan\left(\frac{\ph}{2}\right)\sin&(z-\theta)\partial_z\left[\cos(\theta-z)\partial_\ph+\frac{1}{2}\sin(z-\theta)\left(\cot\left(\frac{\ph}{2}\right)+\tan\left(\frac{\ph}{2}\right)\right)\partial_{\theta}\right.\\
		&\left.+\tan\left(\frac{\ph}{2}\right)\sin(z-\theta)\partial_z\right]
		\end{align*}
		and
		\begin{align*}
		\bar{Y}^2=\sin(\theta-z)\partial_\ph&\left[\sin(\theta-z)\partial_{\ph}+\frac{1}{2}\cos(\theta-z)\left(\cot\left(\frac{\ph}{2}\right)+\tan\left(\frac{\ph}{2}\right)\right)\partial_{\theta}+\right.\\
		&+\left.\tan\left(\frac{\ph}{2}\right)\cos(z-\theta)\partial_z\right]\\
		+\frac{1}{2}\cos(\theta-z)&\left(\cot\left(\frac{\ph}{2}\right)+\tan\left(\frac{\ph}{2}\right)\right)\partial_{\theta}\left[\sin(\theta-z)\partial_{\ph}+\frac{1}{2}\cos(\theta-z)\left(\cot\left(\frac{\ph}{2}\right)\right.\right.\\
		&\left.+\tan\left(\frac{\ph}{2}\right)\right)\partial_{\theta}+\tan\left(\frac{\ph}{2}\right)\cos(z-\theta)\partial_z\bigg]\\
		+\tan\left(\frac{\ph}{2}\right)\cos&(z-\theta)\partial_z\left[\sin(\theta-z)\partial_{\ph}+\frac{1}{2}\cos(\theta-z)\left(\cot\left(\frac{\ph}{2}\right)+\tan\left(\frac{\ph}{2}\right)\right)\partial_{\theta}\right.\\
		&+\tan\left(\frac{\ph}{2}\right)\cos(z-\theta)\partial_z\bigg].
	\end{align*}
	We obtain: \begin{align*}
		\bar{X}^2+\bar{Y}^2&=\partial^2_{\ph,\ph}+\frac{1}{4}\left(\cot\left(\frac{\ph}{2}\right)+\tan\left(\frac{\ph}{2}\right)\right)^2\partial^2_{\theta,\theta}+\left(1+\tan^2\left(\frac{\ph}{2}\right)\right)\partial^2_{\theta,z}\\
		&+\frac{1}{2}\left(\cot\left(\frac{\ph}{2}\right)-\tan\left(\frac{\ph}{ 2}\right)\right)\partial_{\ph}+\tan^2\left(\frac{\ph}{2}\right)\partial^2_{z,z}\\
		&=\partial^2_{\ph,\ph}+\frac{1}{\sin^2(\ph)}\partial^2_{\theta,\theta}+\tan^2\left(\frac{\ph}{2}\right)\partial^2_{z,z}+\frac{1}{\cos^2\left(\frac{\ph}{2}\right)}\partial^2_{\theta,z}+\cot(\ph)\partial_{\ph}.
	\end{align*}
	\section{Computation of the subLaplacian on {$SL(2,\mathbb{R})$}}\label{SubLapSL(2)}
	We now consider the case of $SL(2,\mathbb{R})$. In this case, $(X,Y,Z)$ is satisfying (\ref{LieSL}). In order to use our previous results, we define $(\tilde{X},\tilde{Y},\tilde{Z}):=(iX,iY,-Z)$. This way we get $(\tilde{X},\tilde{Y},\tilde{Z})$ a basis of $\mathfrak{su(2)}$ satisfying (\ref{LieSU}). We have to study \begin{align*}
		\exp\left(xX+yY\right)\exp\left(zZ\right)&\exp\left(\eps\left(aX+bY\right)\right)\\&=\exp\left(-ix\tilde{X}-iy\tilde{Y}\right)\exp\left(-z\tilde{Z}\right)\exp\left(\eps\left(-ia\tilde{X}-ib\tilde{Y}\right)\right).\end{align*}
	In the previous part, as said in the note concerning the two lemmas, the computation can be made using complex coefficients. Then we obtain:
	\begin{align*}
		\exp(xX+yY)\exp(zZ)&\exp(\eps(aX+bY))\\
		&=\exp\left(\ph(\eps)\left(\cos\left(\theta(\eps)\right) \tilde{X}+\sin\left(\theta(\eps)\right)\tilde{Y}\right)\right)\exp\left(z(\eps)\tilde{Z}\right)+\mathcal{O}(\eps^2)\\
		&=\exp\left(i\ph(\eps)\left(\cos\left(\theta(\eps)\right) \tilde{X}+\sin\left(\theta(\eps)\right)\tilde{Y}\right)\right)\exp\left(-z(\eps)\tilde{Z}\right)+\mathcal{O}(\eps^2),\end{align*} with $\ph(0)=-i\ph$, $\theta(0)=\theta$ and $z(0)=-z$.
	
	We get: \begin{equation*}
		\frac{d}{d\eps}\left(i\ph(\eps)\right)_{|\eps=0}=i\ph'(0)=i\left(-ia\cos(\theta-(-z))-ib\sin(\theta-(-z))\right)\\
		=a\cos(\theta+z)+b\sin(\theta+z);\end{equation*}
	\begin{align*}
		\frac{d}{d\eps}\left(-z(\eps)\right)_{|\eps=0}&=-z'(0)=-\tan\left(\frac{-i\ph}{2}\right)\left(-ia\sin(-z-\theta)-ib\cos(-z-\theta)\right)\\
		&=-\tanh\left(\frac{\ph}{2}\right)\left(a\sin(z+\theta)-b\cos(\theta+z)\right);\end{align*}
	\begin{align*}
		\theta'(0)&=\frac{-ia\sin(-z-\theta)-ib\cos(\theta-(-z))}{2}\left(\cot\left(\frac{-i\ph}{2}\right)+\tan\left(\frac{-i\ph}{2}\right)\right)\\
		&=-i\frac{-a\sin(z+\theta)+b\cos(\theta+z)}{2}\left(i\coth\left(\frac{\ph}{2}\right)-i\tanh\left(\frac{\ph}{2}\right)\right)\\
		&=\frac{-a\sin(z+\theta)+b\cos(\theta+z)}{2}\left(\coth\left(\frac{\ph}{2}\right)-\tanh\left(\frac{\ph}{2}\right)\right).
	\end{align*}
	Thus we obtain: \begin{equation*}\bar{X}=\begin{pmatrix} \cos(\theta+z) \\-\frac{\sin(z+\theta)}{2}\left(\coth\left(\frac{\ph}{2}\right)-\tanh\left(\frac{\ph}{2}\right)\right) \\-\tanh\left(\frac{\ph}{2}\right)\sin(z+\theta) \end{pmatrix}\text{ 
	and }\bar{Y}=\begin{pmatrix} \sin(\theta+z) \\\frac{\cos(\theta+z)}{2}\left(\coth\left(\frac{\ph}{2}\right)-\tanh\left(\frac{\ph}{2}\right)\right) \\\tanh\left(\frac{\ph}{2}\right)\cos(z+\theta) \end{pmatrix}.\end{equation*}
	
	As before, we can calculate the subLaplacian operator  \begin{align*}
		L=\frac{1}{2}(\bar{X}^2+\bar{Y}^2)&=\partial^2_{\ph,\ph}+\frac{1}{4}\left(\coth\left(\frac{\ph}{2}\right)-\tanh\left(\frac{\ph}{2}\right)\right)^2\partial^2_{\theta,\theta}+\left(1-\tanh^2\left(\frac{\ph}{2}\right)\right)\partial^2_{\theta,z}\\
		&+\frac{1}{2}\left(\coth\left(\frac{\ph}{2}\right)+\tanh\left(\frac{\ph}{2}\right)\right)\partial_{\ph}+\tanh^2\left(\frac{\ph}{2}\right)\partial^2_{z,z}\\
		&=\partial^2_{\ph,\ph}+\frac{1}{\sinh^2(\ph)}\partial^2_{\theta,\theta}+\tanh^2\left(\frac{\ph}{2}\right)\partial^2_{z,z}+\frac{1}{\cosh^2\left(\frac{\ph}{2}\right)}\partial^2_{\theta,z}+\coth(\ph)\partial_{\ph}.
	\end{align*}
		\section{Proof of Proposition \ref{prop: distance ponctuelle} for {$SL(2,\mathbb{R})$}}\label{subsec: InterpretationDistanceSL}
	The proof is quite similar as for the case of $SU(2)$.
\begin{proof}
	We use the same notations as before.
	%This time we get  \begin{align*}\exp(\ph(\cos(\theta)X+\sin(\theta)&Y))^{-1}\\
% 	&=\begin{pmatrix}
% 			\cosh\left(\frac{\ph}{2}\right)+\sinh\left(\frac{\ph}{2}\right)\cos(\theta) & -\sinh\left(\frac{\ph}{2}\right)\sin(\theta)\\[6pt]
% 			-\sinh\left(\frac{\ph}{2}\right)\sin(\theta) & \cosh\left(\frac{\ph}{2}\right)-\sinh\left(\frac{\ph}{2}\right)\cos(\theta) 
% 		\end{pmatrix}^{-1}\\
% 	&=\begin{pmatrix}
% 			\cosh\left(\frac{\ph}{2}\right)-\sinh\left(\frac{\ph}{2}\right)\cos(\theta) & \sinh\left(\frac{\ph}{2}\right)\sin(\theta)\\[6pt]
% 			\sinh\left(\frac{\ph}{2}\right)\sin(\theta) & \cosh\left(\frac{\ph}{2}\right)+\sinh\left(\frac{\ph}{2}\right)\cos(\theta) 
% 		\end{pmatrix}\\
% 		&=\exp(-\ph(\cos(\theta)X+\sin(\theta)Y)).\end{align*} Using again Lemma \ref{orth} but with $\tilde{X}$, $\tilde{Y}$ and $\tilde{Z}$, we get: \begin{align*}
% 		\exp(\zetaZ)\exp(\ph(\cos(\theta)X+\sin(\theta)Y))&=\exp\left(-z\tilde{Z}\right)\exp\left(-i\ph\left(\cos(\theta)\tilde{X}+\sin(\theta)\tilde{Y}\right)\right)\\
% 		&=\exp\left(-i\ph\left(\cos(\theta-z)\tilde{X}+\sin(\theta-z)\tilde{Y}\right)\right)\exp\left(-z\tilde{Z}\right)\\
% 		&=\exp\big(\ph(\cos(\theta-z)X+\sin(\theta-z)Y)\big)\exp(\zetaZ).
% 	\end{align*}
	Let $g,g'\in SL(2,\mathbb{R})$. This time, the matrix computations give:
	\begin{align*}
		g^{-1}\cdot g'&
% 		=\big(&\exp\big(\ph(\cos(\theta)X+\sin(\theta)Y)\big)\exp(zZ)\big)^{-1}\exp\big(\ph'(\cos(\theta')X+\sin(\theta')Y)\big)\exp(z'Z)\\
% 		=&\exp(-zZ)\exp\big(-\ph(\cos(\theta)X+\sin(\theta)Y)\big)\exp\big(\ph'(\cos(\theta')X+\sin(\theta')Y)\big)\exp(z'Z)\\
% 		=&\exp\big(-\ph(\cos(\theta+z)X+\sin(\theta+z)Y)\big)\exp(-zZ)\exp\big(\ph'(\cos(\theta')X+\sin(\theta')Y)\big)\\
% 		&\exp(z'Z)\\
		=\exp\big(-\ph(\cos(\theta+z)X+\sin(\theta+z)Y)\big)\\
		&\exp\big(\ph'(\cos(\theta'+z)X+\sin(\theta'+z)Y)\big)\exp\big((z'-z)Z\big)
	%\\	=&\exp\big(\rho(\cos(\theta)X+\sin(\theta)Y)\big)\exp(\zetaZ)\exp\big((z'-z)Z\big)
	\end{align*}
	If we denote $M:=\exp\big(\rho(\cos(\Theta)X+\sin(\Theta)Y)\big)\exp(\zeta Z)$, we have:\begin{equation*}M=\begin{pmatrix}
			\cosh\left(\frac{\rho}{2}\right)\cos\left(\frac{\zeta}{2}\right)+\sinh\left(\frac{\rho}{2}\right)\cos\left(\Theta+\frac{\zeta}{2}\right) &-\cosh\left(\frac{\rho}{2}\right)\sin\left(\frac{\zeta}{2}\right)-\sinh\left(\frac{\rho}{2}\right)\sin\left(\Theta+\frac{\zeta}{2}\right)\\[6pt]
			\cosh\left(\frac{\rho}{2}\right)\sin\left(\frac{\zeta}{2}\right)-\sinh\left(\frac{\rho}{2}\right)\sin\left(\Theta+\frac{\zeta}{2}\right) & \cosh\left(\frac{\rho}{2}\right)\cos\left(\frac{\zeta}{2}\right)-\sinh\left(\frac{\rho}{2}\right)\cos\left(\Theta+\frac{\zeta}{2}\right)
	\end{pmatrix}\end{equation*} The matrix $M$ is also equal to the product: \begin{equation*}\exp\big(-\ph(\cos(\theta+z)X+\sin(\theta+z)Y)\big)\exp\big(\ph'(\cos(\theta'+z)X+\sin(\theta'+z)Y)\big).\end{equation*} In particular we get:
	%\begin{equation*}
	%	\begin{pmatrix}
	%		\cosh(\frac{-R}{2})+\sinh(\frac{-R}{2})\cos(\theta+z) & -\sinh(\frac{-R}{2})\sin(\theta+z)\\
	%		-\sinh(\frac{-R}{2})\sin(\theta+z) & \cosh(\frac{-R}{2})-\sinh(\frac{-R}{2})\cos(\theta) 
	%	\end{pmatrix}\begin{pmatrix}
	%	\cosh(\frac{R'}{2})+\sinh(\frac{R'}{2})\cos(\theta'+z) & -\sinh(\frac{R'}{2})\sin(\theta'+z)\\
	%	-\sinh(\frac{R'}{2})\sin(\theta'+z) & \cosh(\frac{R'}{2})-\sinh(\frac{R'}{2})\cos(\theta'+z) 
	%\end{pmatrix}.
	%\end{equation*}
	\begin{align*}
		M_{1,1}&=\cosh\left(\frac{\ph}{2}\right)\cosh\left(\frac{\ph'}{2}\right)+\cosh\left(\frac{\ph}{2}\right)\sinh\left(\frac{\ph'}{2}\right)\cos(\theta'+z)\\
		&-\cosh\left(\frac{\ph'}{2}\right)\sinh\left(\frac{\ph}{2}\right)\cos(\theta+z)-\sinh\left(\frac{\ph}{2}\right)\sinh\left(\frac{\ph'}{2}\right)\cos(\theta'-\theta)\\
		M_{1,2}&=-\cosh\left(\frac{\ph}{2}\right)\sinh\left(\frac{\ph'}{2}\right)\sin(\theta'+z)+\sinh\left(\frac{\ph}{2}\right)\cosh\left(\frac{\ph'}{2}\right)\sin(\theta+z)\\
		&+\sinh\left(\frac{\ph}{2}\right)\sinh\left(\frac{\ph'}{2}\right)\sin(\theta'-\theta)\\
		M_{2,1}&=-\cosh\left(\frac{\ph}{2}\right)\sinh\left(\frac{\ph'}{2}\right)\sin(\theta'+z)+\sinh\left(\frac{\ph}{2}\right)\cosh\left(\frac{\ph'}{2}\right)\sin(\theta+z)\\
		&-\sinh\left(\frac{\ph}{2}\right)\sinh\left(\frac{\ph'}{2}\right)\sin(\theta'-\theta)\\
		M_{2,2}&=\cosh\left(\frac{\ph}{2}\right)\cosh\left(\frac{\ph'}{2}\right)-\cosh\left(\frac{\ph}{2}\right)\sinh\left(\frac{\ph'}{2}\right)\cos(\theta'+z)\\
		&+\cosh\left(\frac{\ph'}{2}\right)\sinh\left(\frac{\ph}{2}\right)\cos(\theta+z)-\sinh\left(\frac{\ph}{2}\right)\sinh\left(\frac{\ph'}{2}\right)\cos(\theta'-\theta)
	\end{align*}
	As $\cosh\left(\frac{\rho}{2}\right)\cos\left(\frac{\zeta}{2}\right)=\frac{M_{1,1}+M_{2,2}}{2}$ and $\cosh\left(\frac{\rho}{2}\right)\sin\left(\frac{\zeta}{2}\right)=\frac{M_{2,1}-M_{1,2}}{2}$, we get:\begin{subnumcases}{}
		\cosh\left(\frac{\rho}{2}\right)\cos\left(\frac{\zeta}{2}\right)=
		\cosh\left(\frac{\ph}{2}\right)\cosh\left(\frac{\ph'}{2}\right)-\sinh\left(\frac{\ph}{2}\right)\sinh\left(\frac{\ph'}{2}\right)\cos(\theta'-\theta) \\
		\cosh\left(\frac{\rho}{2}\right)\sin\left(\frac{\zeta}{2}\right)=
		-\sinh\left(\frac{\ph}{2}\right)\sinh\left(\frac{\ph'}{2}\right)\sin(\theta'-\theta).
	\end{subnumcases}

	\begin{itemize}
		\item This time, the distance between $x:=\Pi_2(g)$ and $y:=\Pi_2(g')$ satisfy:
		\begin{align*}
			\cosh(\rho(x,y))=\cosh(\ph)\cosh(\ph')+\sinh(\ph)\sinh(\ph')\cos(\theta-\theta').
		\end{align*}
		Using the identity $\cosh^2(\frac{\rho}{2})=\frac{1+\cosh(\rho)}{2}$ we obtain: \begin{equation}\label{DistSL}
			\cosh(\rho)=\cosh(\ph)\cosh(\ph')+\sinh(\ph)\sinh(\ph')\cos(\theta-\theta')
		\end{equation} and thus: $\rho=\rho(x,y)$.
		\item Using the same notations as before (this time $N_0$ is the pole used to define polar coordinates on $\mathbf{H}$) and 
		the equivalent of the Heron formula for hyperbolic triangles, we have: \begin{align*}\cos\left(\frac{\mathcal{A}_{y,x,N_0}}{2}\right)&=\frac{1}{4\cosh\left(\frac{\rho}{2}\right)\cosh\left(\frac{\ph}{2}\right)\cosh\left(\frac{\ph'}{2}\right)}(1+\cosh(\ph)+\cosh(\ph')+\cosh(\rho))\\
		&=\cos\left(\frac{\zeta}{2}\right)\end{align*} and thus $\zeta=\textrm{sign}(\zeta) \mathcal{A}_{y,x,N_0}$. Moreover, we have  $\sin\left(\frac{\zeta}{2}\right)=\sin(\theta-\theta')\frac{\sinh\left(\frac{\ph}{2}\right)\sinh\left(\frac{\ph'}{2}\right)}{\cosh\left(\frac{\rho}{2}\right)}$, and so $\zeta<0$ if and only if $\theta<\theta'$.
	\end{itemize}
\end{proof}
\end{appendices}
\bibliographystyle{plain} 

\bibliography{Bibliographie}

\end{document}